\documentclass[10pt,reqno,a4paper]{amsart}
\usepackage[utf8]{inputenc}
\usepackage[hmargin=1in,centering]{geometry}
\usepackage[british]{babel}
\usepackage{amssymb,amstext,eucal,mathrsfs,amscd}
\usepackage{physics}
\usepackage[small]{eulervm} 
\usepackage{tikz}
\usetikzlibrary{cd}
\usepackage{graphicx} 
\usepackage{enumitem} 
\usepackage{cite} 
\usepackage[unicode,pagebackref=true]{hyperref}
\hypersetup{%
  pdftitle   = {Symplectic actions and central extensions},
  pdfkeywords = {symplectic, homogeneous, coadjoint orbit, central extensions, Kostant, Kirillov, Souriau},
  pdfauthor  = {Andrew Beckett,José Figueroa-O'Farrill}
  pdfcreator = {\LaTeX\ with package \flqq hyperref\frqq},
  colorlinks = true,
  linkcolor = blue,
  menucolor = blue,
  urlcolor = [rgb]{0.7,0,0},
  citecolor = [rgb]{0,0.7,0}
}
\renewcommand*{\backref}[1]{}
\renewcommand*{\backrefalt}[4]{%
  \ifcase #1%
  \or [Page~#2.]%
  \else [Pages~#2.]%
  \fi%
}
\theoremstyle{plain}
\newtheorem{lemma}{Lemma}
\newtheorem{proposition}[lemma]{Proposition}
\newtheorem{theorem}[lemma]{Theorem}
\newtheorem{corollary}[lemma]{Corollary}
\newtheorem*{thma}{Theorem~A}
\newtheorem*{thmb}{Theorem~B}
\newtheorem*{thmc}{Theorem~C}
\theoremstyle{definition}
\newtheorem{definition}[lemma]{Definition}
\newtheorem*{remark}{Remark}

\newcommand{\fg}{\mathfrak{g}}
\newcommand{\g}{\mathfrak{g}}
\renewcommand{\k}{\mathfrak{k}}

\newcommand{\fk}{\mathfrak{k}}

\newcommand{\RR}{\mathbb{R}}

\newcommand{\ZZ}{\mathbb{Z}}

\ifdefined\C
\renewcommand{\C}{\mathsf{C}}
\else
\newcommand{\C}{\mathsf{C}}
\fi

\newcommand{\eL}{\mathscr{L}}
\newcommand{\eX}{\mathscr{X}}

\newcommand{\bigO}{\mathcal{O}}
\newcommand{\GL}{\operatorname{GL}}

\newcommand{\Ad}{\operatorname{Ad}}
\newcommand{\ad}{\operatorname{ad}}
\newcommand{\Lie}{\operatorname{Lie}}
\newcommand{\Aff}{\operatorname{Aff}}
\newcommand{\Wedge}{\Lambda}

\newcommand{\pair}[2]{\left\langle #1,#2\right\rangle} 

\DeclareMathOperator{\End}{End}
\DeclareMathOperator{\Hom}{Hom}

\DeclareMathOperator{\pr}{pr}

\newcommand{\MUNCH}[1]{\relax}

\allowdisplaybreaks[1]
\numberwithin{equation}{section}
\begin{document}

\title{Symplectic actions and central extensions}
\author[Beckett]{Andrew Beckett}
\author[Figueroa-O'Farrill]{José Figueroa-O'Farrill}
\address{Maxwell Institute and School of Mathematics, The University
  of Edinburgh, James Clerk Maxwell Building, Peter Guthrie Tait Road,
  Edinburgh EH9 3FD, Scotland, United Kingdom}
\email[AB]{\href{mailto:abeckett@ed.ac.uk}{abeckett[at]ed.ac.uk},
  ORCID: \href{https://orcid.org/0000-0002-7287-3156}{0000-0002-7287-3156}}
\email[JMF]{\href{mailto:j.m.figueroa@ed.ac.uk}{j.m.figueroa[at]ed.ac.uk},
  ORCID: \href{https://orcid.org/0000-0002-9308-9360}{0000-0002-9308-9360}}
\begin{abstract}
  We give a proof of the fact that a simply-connected symplectic
  homogeneous space \((M,\omega)\) of a connected Lie group \(G\) is
  the universal cover of a coadjoint orbit of a one-dimensional
  central extension of \(G\).  We emphasise the rôle of symplectic
  group cocycles and the relationship between such cocycles,
  left-invariant presymplectic structures on \(G\) and central
  extensions of \(G\); in particular, we show that integrability of 
  a central extension of \(\fg\) to a central extension of \(G\) is 
  equivalent to integrability of a representative Chevalley-Eilenberg
  2-cocycle of \(\fg\) to a symplectic cocycle of \(G\).
\end{abstract}
\thanks{EMPG-22-04}
\maketitle
\tableofcontents

\section{Introduction}
\label{sec:intro}

In the study of dynamical systems invariant under the action of a Lie
group $G$, one encounters the notion of an elementary system with
symmetry $G$: namely, a symplectic manifold admitting a transitive
action of $G$ via symplectomorphisms \cite[§14]{MR0260238} (or in
the English translation \cite{MR1461545}). Such spaces are also known
as homogeneous symplectic $G$-spaces. These are the classical
counterparts of free elementary particles with symmetry $G$, which are
given by unitary irreducible representations of $G$.  Such unitary
irreducible representations often result from quantising elementary
systems geometrically.

The paradigmatic examples of such elementary systems are the coadjoint
orbits of $G$.  Let $\g$ be the Lie algebra of $G$ and $\g^*$ its
dual vector space. We shall let $\Ad : G \to \GL(\g)$ denote the
adjoint representation and $\ad: \g \to \End(\g)$ its linearisation
at the identity.  The coadjoint representation of $G$, denoted $\Ad^* : G
\to \GL(\g^*)$, is defined by $\Ad^*_g \alpha = \alpha \circ
\Ad_{g^{-1}}$ for all $\alpha \in \g^*$ and $g \in G$.  Its linearisation at the
identity is the coadjoint representation of the Lie algebra, denoted
by $\ad^* : \g \to \End(\g^*)$  and defined by $\ad^*_X
\alpha = - \alpha \circ \ad_X$ for all $\alpha \in \g^*$ and $X \in
\g$.

Now fix $\alpha \in \g^*$ and define $B_\alpha \in \wedge^2\g^*$
by\footnote{Here and in the sequel we often use the dual pairing
  notation $\left<-,-\right> : \g^* \times \g \to \RR$ instead of
  viewing $\g^*$ as $\Hom(\g,\RR)$.}
$B_\alpha(X,Y) = \left<\alpha,[X,Y]\right>$ for all $X,Y \in \g$.  The
radical of the bilinear form $B_\alpha$ consists of all $X \in \g$
such that $B_\alpha(X,Y) = 0$ for all $Y \in \g$.  This condition is
equivalent to $\ad^*_X\alpha = 0$, so that $X \in \g_\alpha$, the Lie
algebra of the stabiliser
$G_\alpha := \left\{ g \in G~ \middle | ~ \Ad^*_g\alpha = \alpha\right\}$
of $\alpha$. It follows that $B_\alpha$ induces a symplectic inner
product on $\g/\g_\alpha$. Introducing the coadjoint orbit
$\bigO_\alpha := \left\{ \Ad^*_g \alpha ~ \middle | ~ g \in G\right\}$ of
$\alpha$, we have an exact sequence of $G_\alpha$-modules
\begin{equation}\label{eq:KKS-sequence}
  \begin{tikzcd}
    0 \arrow[r] & \g_\alpha \arrow[r] & \g \arrow[r] &
    T_\alpha\bigO_\alpha \arrow[r] & 0,
  \end{tikzcd}
\end{equation}
where $G_\alpha$ acts on $T_\alpha \bigO_\alpha$ via the linear isotropy
representation. This shows that $\g/\g_\alpha \cong
T_\alpha\bigO_\alpha$ and we let $\omega_\alpha \in \wedge^2T^*_\alpha\bigO_\alpha$
be the symplectic inner product on $T_\alpha\bigO_\alpha$ induced from
$B_\alpha$.  Since $\omega_\alpha$ is $G_\alpha$-invariant, the
holonomy principle guarantees that 
$\omega_\alpha$ is the value at $\alpha$ of a $G$-invariant symplectic
form $\omega \in \Omega^2(\bigO_\alpha)$.  This construction is due
independently to Kirillov, Kostant and Souriau and we shall refer to
the symplectic structure on $\bigO_\alpha$ as $\omega_{\mathrm{KKS}}$.

It is not just that coadjoint orbits provide examples of homogeneous
symplectic manifolds, but that in a way to be made precise below, any
homogeneous symplectic manifold is locally isomorphic to a coadjoint
orbit.

Indeed, let $(M,\omega)$ be a symplectic manifold admitting a
transitive left action of a connected Lie group $G$ via
symplectomorphisms.  Let $X \in \g$ and let $\xi_X \in \eX(M)$ denote
the corresponding fundamental vector field.  This defines a Lie
algebra anti-homomorphism $\xi: \g \to \eX(M)$.  Since the symplectic
form $\omega$ is $G$-invariant, the fundamental vector fields are
symplectic: $\eL_{\xi_X}\omega = 0$ and by the Cartan formula, the
one-forms $\imath_{\xi_X}\omega$ are closed.  If $M$ is
simply-connected (although this is typically too strong), these forms
are exact, so that $\imath_{\xi_X}\omega = d\varphi_X$ and we may
always arrange for $\varphi_X$ to be the image of $X \in \g$ under a
linear map $\varphi: \g \to C^\infty(M)$.  Dual to this map we have
the moment map $\mu : M \to \g^*$ defined by
$\left<\mu(p),X\right> = \varphi_X(p)$.  For the case of a coadjoint
orbit $\bigO_\alpha$ of $G$, the moment map relative to the
Kirillov--Kostant--Souriau symplectic structure is simply the
inclusion $i: \bigO_\alpha \to \g^*$.

The moment map relates two spaces on which $G$ acts: the symplectic
manifold $M$ and the dual of the Lie algebra, on which $G$ acts via
the coadjoint representation. A natural question is whether this map
is equivariant. Equivariance of the moment map is obstructed and, as
shown by Atiyah and Bott \cite[Section~6]{MR721448}, the obstructions
are captured by the $G$-equivariant cohomology of $M$: namely, the
moment map is equivariant if and only if the symplectic form admits an
equivariant closed extension.

For $G$ connected, equivariance of the moment map is equivalent to the
comoment map $\varphi : \g \to C^\infty(M)$ being a Lie algebra
homomorphism: $\left\{ \varphi_X,\varphi_Y \right\} = \varphi_{[X,Y]}$
for all $X,Y \in \g$, where $\{-,-\}$ is the Poisson bracket on
$C^\infty(M)$. If this is the case, the image of $M$ under the moment
map is a coadjoint orbit $\bigO$ of $G$.  The triple $(M,\omega,\mu)$
defines an object in the category of hamiltonian $G$-spaces (see
\cite[§5]{MR0294568}), where a morphism between two objects
$(M,\omega,\mu)$ and $(M',\omega',\mu')$ is a smooth map
$\phi: M \to M'$ such that $\phi^*\omega'=\omega$ and
$\phi^*\mu' = \mu$. Kostant \cite[Proposition 5.1.1]{MR0294568} proves
that $\phi$ is a $G$-equivariant covering. Since an equivariant moment
map $\mu : M \to \bigO$ obeys $\mu^*\omega_{\mathrm{KKS}} = \omega$,
it defines a morphism between the hamiltonian $G$-spaces
$(M,\omega,\mu)$ and $(\bigO,\omega_{\mathrm{KKS}},i)$, where
$i : \bigO \to \g^*$ is the inclusion. It follows that
$\mu : M \to \bigO$ is a covering. If $M$ is simply-connected, it is
(symplectically as well as topologically) the universal cover of a
coadjoint orbit of $G$.

What about if $\mu$ is not equivariant? Souriau\cite[§11]{MR0260238}
showed that $\mu$ is \emph{always} equivariant relative to an
affinisation of the coadjoint representation and hence if $M$ is
simply-connected, it is the universal cover of an affine orbit of $G$
in $\g^*$.

To describe this, consider the failure of equivariance
$\vartheta(g,p):= \Ad^*_g \mu(p) - \mu(g \cdot p)$, which defines a
smooth function $\vartheta : G \times M \to \g^*$.  It happens that
for $M$ connected, $\vartheta$ does not depend on $p$ (see
Lemma~\ref{lemma:symp-theta-cocycle}) and hence, letting $\pr_1 :
G\times M \to G$ denote the cartesian projection, $\vartheta =
\pr_1^*\theta$ for some smooth group cocycle $\theta: G \to \g^*$.
The cocycle condition
\begin{equation}
  \theta(g_1 g_2) = \Ad^*_{g_1} \theta(g_2) + \theta(g_1)
  \qquad\text{for all $g_1,g_2 \in G$}
\end{equation}
allows us to define a Lie group homomorphism $\rho : G \to \Aff(\g^*)$
by $\rho(g) \alpha = \Ad^*_g \alpha - \theta(g)$ (the sign is conventional), relative to
which $\mu$ is equivariant by construction: $\mu(g\cdot p) =
\rho(g) \mu(p)$.

The group cocycle $\theta$ linearises at the identity to $d_e\theta:
\g \to \g^*$, which defines a Chevalley--Eilenberg $2$-cocycle $c \in
\wedge^2\g^*$ via
\begin{equation}
  c(X,Y) = \left<(d_e\theta)X,Y\right>, \qquad\text{for all $X,Y \in \g$.}
\end{equation}
(See Lemma~\ref{lemma:Dtheta-symplectic-cocycle}.)  Following Souriau,
one says that $\theta$ is a symplectic cocycle.  It follows that
$c(X,Y) = \left\{ \varphi_X,\varphi_Y \right\} - \varphi_{[X,Y]}$,
which is actually constant.  If $[c] = 0 \in H^2(\g)$, so that
$c(X,Y) = -\left<\mu_0, [X,Y]\right>$ for some $\mu_0 \in\g^*$, we can
modify the moment map to $\mu' = \mu - \mu_0$, which is now
equivariant with respect to the coadjoint action: $\mu'(g\cdot p) =
\Ad_g^*\mu'(p)$.

If $[c] \neq 0$ in cohomology, it defines a nontrivial one-dimensional
central extension $\widehat\g$ of $\g$. By the Lie correspondence we
get a one-dimensional central extension $\widehat G$ of the universal
cover of $G$, but it is not a priori clear that we get a central
extension of $G$ itself. In fact, there are known obstructions to a
central extension of the Lie algebra of a Lie group to integrate to a
central extension of that Lie group \cite{MR948561,MR1424633}.

The purpose of this paper is to show that such obstructions are
overcome for symplectic actions of connected Lie groups. More
precisely, we give a proof of Theorem~A below
(see~Theorem~\ref{thm:homo-sym-ext}).

\begin{thma}
  Let $(M,\omega)$ be a simply-connected symplectic manifold admitting
  a transitive action of a connected Lie group $G$ via
  symplectomorphisms.  Then $(M,\omega)$ is the universal cover of a
  coadjoint orbit of a one-dimensional central extension of $G$.
\end{thma}

This theorem has acquired an almost ``folkloric'' quality in that,
despite being quoted often, to our knowledge there is no proof in the
literature which does not make additional and, from our perspective,
unwarranted assumptions.  In fact, in a recent paper
\cite{donato_iglesias-zemmour_2021}, in which it is shown that, in the
diffeological category, every symplectic manifold is a coadjoint orbit
of its group of hamiltonian diffeomorphisms, one can read that
\begin{quotation}
  the optimal result in the category of manifolds states that the
  symplectic manifold is, up to covering, an affine coadjoint orbit of
  the group.
\end{quotation}

The question of the construction of the central extension of $G$ from
a symplectic cocycle was addressed in \cite{MR2587386}. There is no
analogue of Theorem~A in that paper, although they outline a procedure
to integrate the symplectic cocycle to a central extension of the Lie
group and illustrate it with two examples.  To our knowledge, the first
published statement of Theorem~A is by Kirillov
\cite[§15]{MR0412321}. (See also the more recent book
\cite{MR2069175}.) Theorem~1 in Section~15.2 of \cite{MR0412321} (or
Proposition~4 in Section~1.4 of \cite{MR2069175}) states a roughly
equivalent result:

\begin{thmb}
  Every homogeneous symplectic manifold whose group of motions is a
  connected Lie group is locally isomorphic to an orbit in the
  coadjoint representation of the group $G$ or a central extension of
  $G$ with the aid of $\RR$.
\end{thmb}

The statement of the theorem notwithstanding, a closer look at the
discussion preceding the theorem reveals that this is not the result
which is proved.  (The same applies to the discussion in
\cite[Section~1.4]{MR2069175}.)  Kirillov starts with $G$ connected,
but then passes to the universal covering group.  He then considers a
central extension $\g_1$ of the Lie algebra $\g$ of $G$, giving a
central extension $G_1$ of the universal cover of $G$.  He then shows
that $M$ is a homogeneous space of $G_1$.  Hence the theorem which is
actually proved in \cite[§15.2]{MR0412321} could be paraphrased as
follows:

\begin{thmc}
  Every homogeneous symplectic manifold whose group of motions is a
  connected \emph{and simply-connected} Lie group $G$ is locally
  isomorphic to an orbit in the coadjoint representation of $G$ or a
  one-dimensional central extension of $G$.
\end{thmc}

In fact, the assumption that $G$ is simply-connected is common in the
literature, appearing also in Kostant \cite[§5]{MR0294568}, Chu
\cite{MR342642} and Sternberg \cite{MR379759}, all of whom prove
equivalent versions of Theorem~C.

One might be forgiven for thinking that nothing is lost by assuming ab
initio that $G$ is simply-connected: after all, if a Lie group acts
transitively on a manifold, so does its universal cover. Alas, we
would like to argue that that the topology of $G$ matters. In
applications, one is often trying to classify simply-connected
homogeneous symplectic manifolds of a given Lie group $G$ and
Theorem~A says that we need to consider coadjoint orbits of a
one-dimensional central extension of $G$.  Such extensions are
classified by the group cohomology $H^2_{\text{loc}}(G;\RR)$ (the
topology of the kernel of the central extension ends up being of no
consequence) whose cochains are smooth near the identity.  Even if we
restrict to the smooth group cohomology group $H^2(G;\RR)$, which
classifies central extensions that are diffeomorphic to a product
$G \times \RR$, the van~Est theorem \cite{MR0059285} (see also the
recent clear exposition in \cite{MR4081118}) gives an isomorphism
$H^2(G;\RR) \cong H^2(\g,\k)$, where $H^2(\g,\k)$ is the
Chevalley--Eilenberg cohomology of $\g$ relative to the Lie algebra
$\k$ of the maximal compact subgroup $K \subset G$. Since the maximal
compact subgroups of $G$ and of its universal cover need not coincide,
the choice of $G$ does matter. Therefore Theorem~C is weaker than what
is needed in applications and we require the full strength of
Theorem~A.

The question of whether a central extension of the Lie algebra $\g$ of
a Lie group $G$ integrates to a central extension of $G$ has been
studied by Tuynman and Wiegerinck \cite{MR948561} and later by Neeb
\cite{MR1424633}, on whose results we rely.  For the case of a
one-dimensional central extension $\widehat\g$ of $\g$ with
Chevalley--Eilenberg cocycle $c \in \wedge^2\g^*$, one can prove that
$\widehat\g$ integrates to a one-dimensional central extension of $G$
if and only if the left-invariant presymplectic structure
$\Omega \in \Omega^2(G)$ associated to the cocycle $c$ admits a
``moment map'' for the left action of $G$ on itself; that is, if for
every right-invariant vector field $\rho_X \in \eX(G)$, the closed
one-form $\imath_{\rho_X}\Omega$ is exact.  It is this result which we
shall exploit to show that the affine orbit of $G$ in $\g^*$ arising
as the image of the moment map of a symplectic $G$-action is a
(linear) coadjoint orbit of a one-dimensional central extension
$\widehat G$ of $G$.  We will do so by showing that the symplectic
cocycle $\theta : G \to \g^*$ gives us the desired moment map.

In the case where the symplectic manifold $(M,\omega)$ is
pre-quantisable, so that (some multiple of) $\omega$ defines an
integral class in de~Rham cohomology, Neeb and Vizman \cite{MR2600913}
construct the central extension $\widehat G$ of $G$ by pulling back
the pre-quantum bundle to the group via the orbit map $G \to M$.

This paper is organised as follows. In Section~\ref{sec:symp-action}
we review the basic definitions of Lie group actions on symplectic
manifolds, moment maps and hamiltonian $G$-spaces. We review the
conditions for the equivariance of the moment map
(Proposition~\ref{prop:hamiltonian-action}) and state a theorem of
Kostant's (Theorem~\ref{thm:kostant-covering}) concerning the nature
of morphisms between hamiltonian $G$-spaces. We also review the
notions of symplectic cocycles and the corresponding
Chevalley--Eilenberg cocycles and discuss the Lie algebra extension
resulting from a non-equivariant moment map. In
Section~\ref{sec:symp-cocycle} we go deeper into the relation between
symplectic cocycles of a Lie group and their corresponding
Chevalley--Eilenberg cocycles. In Section~\ref{sec:group-ext} we study
one-dimensional central extensions of a Lie group $G$ and its
coadjoint orbits, which can be interpreted as affine orbits of $G$ on
$\g^*$. We show how such a central extension gives rise to a
symplectic cocycle of $G$ which agrees with the one defined by the
affine orbits. We end the section recording the theorem of Neeb
(Theorem~\ref{thm:neeb}) on the conditions for a central extension of
$\g$ to integrate to a central extension of $G$.
Section~\ref{sec:sympl-struct-central-extn} contains our proof of
Theorem~A (Theorem~\ref{thm:homo-sym-ext}), which we then reinterpret
in a number of ways: in terms of actions of the original group 
\(G\) instead of its central extension
(Section~\ref{sec:interpretation-G-action}); in terms of affine orbits
in \(\fg^*\) rather than coadjoint orbits of the central extension
(Section~\ref{sec:interpretation-affine-G-action}); and lastly in terms 
of left-invariant presymplectic structures on Lie groups
(Section~\ref{sec:group-presymp}). 
Finally, in Section~\ref{sec:exact-symp}, we show that the problem of
existence of an appropriate central group extension simplifies
significantly when one considers \emph{exact} symplectic symmetric
spaces; in particular, the central extension is geometrically
trivial (diffeomorphic to a product) and can be explicitly constructed.
We end the paper with two appendices on Lie groups, Lie algebras and 
their cohomology.

\section*{Acknowledgements}
\label{sec:acknowledgements}

This work arose out of a project of the second-named author (JMF) with
Ross Grassie and Stefan Prohazka in which we set out to describe
elementary systems with Lifshitz symmetry. Theorem~A simplifies some
of the calculations because it would allow ignoring some of the (many)
central extensions of the Lifshitz Lie algebras, but we were unable to
find a published proof of that theorem.  JMF takes pleasure in
thanking them for many useful discussions.  JMF would also like to
thank Dieter Van den Bleeken for helpful correspondence on this topic.
The research of AB is partially funded by an STFC postgraduate
studentship.

\section{Symplectic and hamiltonian actions}
\label{sec:symp-action}

\subsection{Basic definitions}
\label{sec:symp-action-basics}

Let \((M,\omega)\) be a symplectic manifold equipped with a transitive,
symplectic (left) action by a connected Lie group \(G\). It follows that
\(M\) must also be connected. Each group element \(g\in G\) defines a 
symplectomorphism \(p\mapsto g\cdot p\) of \(M\), the pushforward and 
pullback of which we denote by \(g_*\) and \(g^*\) respectively. 
The \emph{fundamental vector field associated with \(X\in\fg=\Lie(G)\)}
is the vector field \(\xi_X\) given by
\((\xi_X)_p = \dv{t} \exp(tX)\cdot p\eval_{t=0}\). This assignment
defines a linear map \(\xi:\fg\to\eX(M)\) given by \(X\mapsto\xi_X\)
which, by a straightforward calculation, is equivariant in the sense
that \(g_*\xi_X=\xi_{\Ad_gX}\). Infinitesimally, this gives
\(\comm{\xi_X}{\xi_Y}=-\xi_{\comm{X}{Y}}\) for all \(X,Y\in\fg\), so
\(\xi\) is a Lie algebra anti-homomorphism.

The action being symplectic says that \(g^*\omega=\omega\) for all
\(g\in G\); infinitesimally, \(\eL_{\xi_X}\omega=0\) for all
\(X\in\fg\). Using \(d\omega=0\) and Cartan's magic formula
\(\eL_{\xi_X}\omega=d\imath_{\xi_X}\omega+\imath_{\xi_X}d\omega\), we
see that \(\imath_{\xi_X}\omega\) is a closed one-form. Suppose that
it is actually exact for all \(X\in\fg\); that is, there exist some
functions \(\varphi_X\in C^\infty\) such that
\begin{equation}\label{eq:symp-ham-vfs}
  \imath_{\xi_X} \omega = d\varphi_X.
\end{equation}
Note that if \(M\) is simply-connected, as we will later assume, then
such functions do indeed exist. Each \(\varphi_X\) is uniquely defined
only up to addition of a constant, so we may assume without
loss of generality that the map \(\varphi:\fg\to C^\infty(M)\) sending
\(X\mapsto\varphi_X\) is linear. It is then
uniquely defined only up to addition of a linear map \(\fg\to\RR\).

The symplectic form \(\omega\) induces a Poisson bracket on
\(C^\infty(M)\): for \(f\in C^\infty(M)\), we let \(\chi_f\) be the
hamiltonian vector field associated to \(f\), i.e., the vector field
satisfying \(\imath_{\chi_f}\omega=df\), which exists and
is unique by the nondegeneracy of \(\omega\). The Poisson bracket
of \(f,g\in C^\infty(M)\) is given by
\begin{equation}
	\acomm{f}{g} = \omega(\chi_f,\chi_g) = -\eL_{\chi_f}g.
\end{equation}
In particular, note that \(\chi_{\varphi_X}=\xi_X\) and
\begin{equation}
	\acomm{\varphi_X}{\varphi_Y} = \omega(\xi_X,\xi_Y) = -\eL_{\xi_X}\varphi_Y.
\end{equation}

\subsection{The moment map and associated cocycles}
\label{sec:moment-map-cocycle}

We define a map \(\mu:M\to\fg^*\) known as the \emph{moment map}
associated to the symplectic action by
\begin{equation}\label{eq:moment-map}
	\pair{\mu(p)}{X} = \varphi_X(p),
\end{equation}
where \(\pair{-}{-}:\fg^*\times \fg\to\RR\) is the natural pairing.
Dually, \(\varphi\) is referred to as the \emph{comoment
  map}. Changing the comoment map by the addition of a linear map
\(\fg\to \RR\) (i.e., an element of \(\fg^*\)) corresponds to changing
\(\mu\) by the addition of the same element of \(\fg^*\).

Recalling that \(\fg^*\) comes equipped with a coadjoint action
\(\Ad^*: G\to\GL(\fg^*)\), it is natural to ask whether \(\mu\) is (or
can be defined to be) equivariant. We thus define a map
\(\vartheta:G\times M\to \fg^*\) by
\begin{equation}\label{eq:vartheta-def}
	\vartheta(g,p) = \Ad^*_g\mu(p) - \mu(g \cdot p)
\end{equation}
which measures the failure of \(\mu\) to be equivariant. 

\begin{lemma}\label{lemma:symp-theta-cocycle}
	The map \(\vartheta\) is independent of \(p\in M\) and so \(\vartheta=\pr_1^*\theta\) for some map \(\theta:G\to\fg^*\), where \(\pr_1:G\times M\to G\) is the projection to the first factor. Moreover, \(\theta\) satisfies
	\begin{equation}\label{eq:symp-theta-cocycle}
		\theta(g_1g_2) = \Ad^*_{g_1}\theta(g_2) + \theta(g_1).
	\end{equation}
\end{lemma}

\begin{proof}
	For the first claim, let \(\Theta_{g,X}\in C^\infty(M)\) be the function given by \(\Theta_{g,X}(p)=\pair{\vartheta(g,p)}{X}\) for \(g\in G\) and \(X\in\fg\). We will show that \(\Theta_{g,X}\in C^\infty(M)\) is locally constant (and therefore constant, since \(M\) is connected). We have
	\begin{equation}
		\Theta_{g,X}(p)
			= \pair{\Ad_g^*\mu(p)}{X} - \pair{\mu(g\cdot p)}{X}
			= \varphi_{\Ad_{g^{-1}}X}(p)-\varphi_X(g\cdot p).
	\end{equation}
	so we see that \(\Theta_{g,X}=\varphi_{\Ad_{g^{-1}}X}-g^*\varphi_X\).
	Then
	\begin{equation}
		d\Theta_{g,X}
			= d\varphi_{\Ad_{g^{-1}}X} - g^*d\varphi_X
			= \imath_{\xi_{\Ad_{g^{-1}}X}}\omega - g^*\imath_{\xi_X}\omega
	\end{equation}
	so pairing with an arbitrary vector field \(\eta\in\eX(M)\), we find
	\begin{equation}
		(d\Theta_{g,X})(\eta) = \omega(\xi_{\Ad_{g^{-1}}X},\eta) - \omega(\xi_X,g_*\eta) = \omega(\xi_{\Ad_{g^{-1}}X},\eta) - (g^*\omega)(g^{-1}_*\xi_X,\eta)
	\end{equation}
	but since the action of \(G\) on \((M,\omega)\) is symplectic (\(g^*\omega=\omega\)) and \((g^{-1})_*\xi_X = \xi_{\Ad_{g^{-1}}X}\), this vanishes, so indeed \(d\Theta_{g,X}=0\). We thus have a map \(\theta:G\to\fg^*\) given by \(\theta(g)=\Ad^*_g\mu(p) - \mu(g \cdot p)\), or
	\begin{equation}\label{eq:theta-phi-reln}
		\pair{\theta(g)}{X} 
		= \varphi_{\Ad_{g^-1}X}(p) - \varphi_X(g\cdot p)
	\end{equation}
	for arbitrary \(p\in M\). For the second claim, we now compute directly
	\begin{equation}
	\begin{split}
		\theta(g_1g_2) 
		&= \Ad^*_{g_1g_2}\mu(p)-\mu(g_1g_2p)\\
		&= \Ad^*_{g_1}\Ad^*_{g_2}\mu(p) - \mu(g_1(g_2p))\\
		&= \Ad^*_{g_1}\qty(\theta(g_2)-\mu(g_2p)) - (\Ad^*_{g_1}\mu(g_2p)-\theta(g_1))\\
		&= \Ad^*_{g_1}\theta(g_2) + \theta(g_1)
	\end{split}
	\end{equation}
so \(\theta\) satisfies \eqref{eq:symp-theta-cocycle} as claimed.
\end{proof}

Equation~\eqref{eq:symp-theta-cocycle} says that \(\theta\) is a 1-cocycle for \(G\) with values in the coadjoint representation (see Appendix~\ref{sec:group-cohomology}). The following result shows in particular that \(\theta\) is a \emph{symplectic} group cocycle (see Section~\ref{sec:symp-cocycle}).

\begin{lemma}\label{lemma:Dtheta-symplectic-cocycle}
	The derivative of \(\theta\) at the identity is a linear map
        \(d_e\theta: \fg \to \fg^*\). The bilinear form \(c\) on
        \(\fg\) defined by \(c(X,Y)=\pair{d_e\theta(X)}{Y}\) is
        alternating (i.e., it lies in \(\Wedge^2\fg^*\)) and is a
        Chevalley--Eilenberg cocycle.
\end{lemma}

\begin{proof}
	The first claim is immediate from the natural identifications \(\fg=T_eG\) and \(T_0\fg^*\cong\fg^*\). Now let \(X,Y\in\fg\) and \(p\) be an arbitrary point in \(M\); using equation~\eqref{eq:theta-phi-reln} gives
	\begin{equation}
	\begin{split}\label{eq:c-phi-rel}
		c(X,Y) = \pair{d_e\theta(X)}{Y}
		&= \dv{t}\qty(\varphi_{\Ad_{\exp(-tX)}Y}(p)-\varphi_Y(\exp(tX)p))\eval_{t=0}\\
	 	&= \qty(-\varphi_{\comm{X}{Y}}-\eL_{\xi_X}\varphi_Y)(p)\\
	 	&= \qty(-\varphi_{\comm{X}{Y}}+\acomm{\varphi_X}{\varphi_Y})(p).
	\end{split}
	\end{equation}
	Note that the LHS is independent of the point \(p\in M\), so \(\varphi_{\comm{X}{Y}}-\acomm{\varphi_X}{\varphi_Y}\) is constant on \(M\). Clearly this is skew-symmetric in \(X,Y\), hence \(c\in\Wedge^2\fg^*\). Finally, we must show that \(\partial_{\mathrm{CE}}c=0\). This follows from the expression for \(c\) above and the Jacobi identities for \(\fg\) and for the Poisson bracket on \(C^\infty(M)\); alternatively, it is a consequence of \(\theta\) being a symplectic cocycle (Lemma~\ref{lemma:symp-cocycles}).
\end{proof}

We now record a result linking the cohomology classes of the cocycles \(\theta\) and \(c\) to each other and to properties of the (co)moment map. This is a standard result, so we will not belabour the point.

\begin{proposition}\label{prop:hamiltonian-action}
	Since \(G\) is a connected Lie group, the following are equivalent:
	\begin{enumerate}
		\item The moment map \(\mu\) can be chosen to be equivariant;\label{item:mu-equi}
		\item The cohomology class of \(\theta\) is trivial;\label{item:theta-triv}
		\item The comoment map \(\varphi\) can be chosen to be a Lie algebra homomorphism;\label{item:phi-hom}
		\item The cohomology class of \(c\) is trivial.\label{item:c-triv}
	\end{enumerate}
	If any (and therefore all) of these conditions hold, the action of \(G\) on \(M\) is said to be \emph{hamiltonian}.
\end{proposition}

\begin{proof}
	(\eqref{item:mu-equi}\(\iff\)\eqref{item:theta-triv}) If \(\mu\) is equivariant then clearly \(\theta=0\). Conversely, if \(\theta=\partial\mu_0\) for some \(\mu_0\in\fg\), for all \(g\in G\) we have \(\theta(g)=\Ad^*_g\mu_0-\mu_0\). But recall that \(\theta(g)=\vartheta(g,p)=\Ad^*_g\mu(p)-\mu(g\cdot p)\) for arbitrary \(p\in M\), so letting \(\mu'(p)=\mu(p)-\mu_0\), we have \(\Ad^*_{g}\mu'(p)=\mu'(g\cdot p)\).
		
	(\eqref{item:phi-hom}\(\iff\)\eqref{item:c-triv}) By \eqref{eq:c-phi-rel}, we have \(c(X,Y) = \acomm{\varphi_X}{\varphi_Y} -  \varphi_{\comm{X}{Y}}\) for all \(X,Y\in \fg\).	Clearly then, \(c=0\) if \(\varphi\) is a Lie algebra homomorphism. Conversely, if \(c=\partial_{\mathrm{CE}}\mu_0\), we have \(c(X,Y)=-\mu_0(\comm{X}{Y})\). Then setting \(\varphi'=\varphi-\mu_0\), for all \(X,Y\in\fg\), we have
	\begin{equation}
	\begin{split}
		\acomm{\varphi'_X}{\varphi'_Y} - \varphi'_{\comm{X}{Y}}
		&= \acomm{\varphi_X-\mu_0(X)}{\varphi_Y-\mu_0(Y)} - (\varphi_{\comm{X}{Y}} - \mu_0(\comm{X}{Y}))\\
		&= \acomm{\varphi_X}{\varphi_Y} - \varphi_{\comm{X}{Y}} + \mu_0(\comm{X}{Y})\\
		&= c(X,Y)+\mu_0(\comm{X}{Y})\\
		&= 0.
	\end{split}
	\end{equation}
	
	(\eqref{item:mu-equi}\(\implies\)\eqref{item:phi-hom}) Let \(p\in M\) and \(X\in\fg\). We have \(	\mu(\exp(tX)p)=\Ad^*_{\exp(tX)}\mu(p)\); pairing with \(Y\in\fg\) yields
	\begin{equation}
		\pair{\mu(\exp(tX)p)}{Y} = \pair{\mu(p)}{\Ad_{\exp(-tX)}Y}
	\end{equation}
	or equivalently
	\begin{equation}
		\varphi_Y(\exp(tX)p)=\varphi_{\Ad_{\exp(-tX)}Y}(p) = \varphi_{\exp(-t\ad_X)Y}(p).
	\end{equation}
	Differentiating this expression at \(t=0\) yields \(\eL_{\xi_X}\varphi_Y(p) = -\varphi_{\ad_XY}(p)\). But this is just
	\begin{equation}
		\acomm{\varphi_X}{\varphi_Y} = \varphi_{\comm{X}{Y}}
	\end{equation}
	after abstracting away \(p\) and eliminating a sign.
	
	(\eqref{item:phi-hom}\(\implies\)\eqref{item:mu-equi}) Since \(G\) is connected, it is sufficient to prove that \(\mu\) is infinitesimally equivariant. For all \(p\in M\) and \(X\in\fg\),
	\begin{equation}
		\pair{d_p\mu(\xi_X)}{Y}
		= d_p\varphi_Y(\xi_X)
		= -\acomm{\varphi_X}{\varphi_Y}(p)
		= -\varphi_{\comm{X}{Y}}(p)
		= -\pair{\mu(p)}{\comm{X}{Y}}
		= \pair{\ad^*_X\mu(p)}{Y},
	\end{equation}
	so \(d\mu(\xi_X)=\ad_X^*\mu\), which is the infinitesimal form of \(\mu(g\cdot p)=\Ad^*_g\mu(p)\).
\end{proof}

\begin{remark}
  The first three parts of the preceding proof do not
  require that \(G\) be connected. Indeed, the only place where we
  have used this hypothesis in any of the above discussion is in
  the fourth part of the proof above, and implicitly in the proof of
  Lemma~\eqref{lemma:symp-theta-cocycle}, although that only required
  the weaker assumption that \(M\) is connected.
  It is already clear, then, that the symplectic cocycle
  \(\theta\) captures important information about (co)moment maps in a
  more general setting, although there is more nuance if \(G\) is not
  connected. We will however require that \(G\) is connected in later 
  sections, and we will also need to assume that \(M\) is simply-connected.
  We also note that in the case where \(G\) is connected, instead of proving
  that \eqref{item:mu-equi}\(\iff\)\eqref{item:phi-hom} as above, we could
  instead invoke a general result about symplectic group cocycles
  (Proposition~\ref{prop:symp-cocycles-triv}) to show that
  \eqref{item:theta-triv}\(\iff\)\eqref{item:c-triv}.
\end{remark}

If the $G$ action on a symplectic manifold $(M,\omega)$ is both 
hamiltonian with moment map $\mu$ and transitive, we say that
$(M,\omega,\mu)$ is a \emph{hamiltonian $G$-space}. Hamiltonian
$G$-spaces form a category whose morphisms 
$(M,\omega,\mu) \to (M',\omega',\mu')$ consist of smooth maps 
$\phi: M \to M'$ such that $\phi^*\omega'= \omega$ and
$\phi^* \mu' = \mu$.  The paradigmatic examples of hamiltonian
$G$-spaces are the coadjoint orbits.  Let $\bigO \subset \g^*$ be a
coadjoint orbit with the Kirillov--Kostant--Souriau symplectic structure
$\omega_{\mathrm{KKS}}$. The coadjoint action of $G$ on $\bigO$ is
hamiltonian with moment map $i: \bigO \to \g^*$ the inclusion, making
$(\bigO,\omega_{\mathrm{KKS}}, i)$ into a hamiltonian $G$-space.

We record for later use the following result of Kostant's
\cite[Prop.~5.1.1]{MR0294568}.

\begin{theorem}[Kostant \cite{MR0294568}]\label{thm:kostant-covering}
  Let $\phi: M \to M'$ be a morphism of hamiltonian $G$-spaces
  $(M,\omega,\mu)$ and $(M',\omega',\mu')$, with $G$ a connected Lie
  group.  Then $\phi$ is $G$-equivariant and a covering map.
\end{theorem}

Let $(M,\omega,\mu)$ be a hamiltonian $G$-space with $G$ connected, so
that $M$ too is connected.  Then the image of the moment map is a
coadjoint orbit $\bigO \subset \g^*$ and we may restrict the codomain
so that $\mu : M \to \bigO$.

\begin{proposition}\label{prop:moment-map-is-morphism}
  The moment map $\mu : M \to \bigO$ defines a morphism of hamiltonian
  $G$-spaces $(M,\omega,\mu)$ and $(\bigO,\omega_{\mathrm{KKS}},i)$.
\end{proposition}

\begin{proof}
  We need only show that $\mu^*\omega_{\mathrm{KKS}} = \omega$.  Let $o
  \in M$ with $\mu(o) = \lambda \in \g^*$ and let $\bigO$ be the
  coadjoint orbit of $\lambda$.  Consider the following triangle
  \begin{equation}\label{eq:M-G-O-mu-triangle}
    \begin{tikzcd}
      & G \arrow[dl,"a_o"'] \arrow[dr,"a_\lambda"] &\\
      M \arrow[rr,"\mu"] & & \bigO
    \end{tikzcd}
  \end{equation}
  where $a_o : G \to M$, sending $g \mapsto g \cdot o$, and
  $a_\lambda : G \to \bigO$, sending $g \mapsto \Ad_g^*
  \lambda$, are the orbit maps.  All maps are surjections and, since
  $\mu$ is $G$-equivariant, the triangle commutes; that is, $\mu \circ
  a_o = a_\lambda$.

  Since both $\omega_{\mathrm{KKS}}$ and $\omega$ are $G$-invariant, it
  is enough to show that $\mu^*\omega_{\mathrm{KKS}}$ and $\omega$ agree
  at $o \in M$. For every $X \in \g$, we let $\zeta_X$ be the fundmental
  vector field on \(\bigO\); we then have $\left( \xi_X  \right)_o =
  (a_o)_* X$ and $\left( \zeta_X \right)_\lambda = (a_\lambda)_* X$. The
  images of $\g$ under these maps span $T_oM$ and $T_\lambda\bigO$,
  respectively, since the orbit maps are surjections.  Hence it is
  enough to show that $\mu^*\omega_{\mathrm{KKS}}$ and $\omega$ agree on
  such tangent vectors.  Also, by the chain rule and the commutativity
  of the diagram,
  \begin{equation}
    \mu_* \left(\xi_X\right)_o = (\mu \circ a_o)_* X = (a_\lambda)_* X = \left(\zeta_X\right)_\lambda.
  \end{equation}
  Therefore on the one hand, we have
  \begin{equation}
    (\mu^*\omega_{\mathrm{KKS}})(\xi_X,\xi_Y)(o) = \omega_{\mathrm{KKS}}(\mu_*\xi_X,\mu_*\xi_Y)(\lambda)\\
                                              = \omega_{\mathrm{KKS}}(\zeta_X,\zeta_Y)(\lambda)\\
                                              = \left<\lambda,[X,Y]\right>,
  \end{equation}
  and on the other hand,
  \begin{equation}
    \omega(\xi_X,\xi_Y)(o) = \left\{ \varphi_X,\varphi_Y \right\}(o)\\
                         = \varphi_{[X,Y]}(o)\\
                         = \left<\mu(o),[X,Y]\right>\\
                         = \left<\lambda,[X,Y]\right>.
  \end{equation}
\end{proof}

As an immediate corollary of this result and Kostant's
Theorem~\ref{thm:kostant-covering}, we have the following.

\begin{corollary}\label{cor:covering}
  Let $G$ be a connected Lie group and $(M,\omega,\mu)$ a
  simply-connected hamiltonian $G$-space.  Then $M$ is (symplectically
  as well as topologically) the universal cover of a coadjoint orbit
  of $G$.
\end{corollary}

\begin{remark}
  Another immediate corollary of
  Proposition~\ref{prop:moment-map-is-morphism} is that
  \begin{equation}
    \label{eq:presymp-G-agree}
    a_o^* \omega = a_o^* \mu^* \omega_{\mathrm{KKS}}= (\mu
    \circ a_o)^* \omega_{\mathrm{KKS}} = a_\lambda^*
    \omega_{\mathrm{KKS}}.
  \end{equation}
  So that the two left-invariant $2$-forms on $G$ resulting by pulling
  back the symplectic forms on $M$ and $\bigO$ via the orbit maps
  coincide.
\end{remark}

\begin{remark}
  In the general case, where the action is not necessarily
  hamiltonian, we can define an affine (not linear) action
  \(\rho:G\to\Aff(\fg^*)\) by
  \begin{equation}\label{eq:theta-aff-act}
    \rho(g)\alpha = \Ad^*_g\alpha - \theta(g).
  \end{equation}
  for \(g\in G\), \(\alpha\in \fg^*\). Indeed,
  the cocycle condition \eqref{eq:symp-theta-cocycle} is equivalent to
  \(g_1\cdot(g_2\cdot\alpha)=(g_1g_2)\cdot\alpha\). By the definition
  of \(\theta\), the moment map \(\mu\) is equivariant with respect to
  this action:
  \begin{equation}
    \mu(g\cdot p) = \rho(g) \mu(p).
  \end{equation}
  In Section~\ref{sec:lie-alg-ext}, we will show that this affine
  action can be viewed as the restriction to a hyperplane of a linear
  action of $G$ on the dual $\widehat\g^*$ of a central extension
  \(\widehat\fg\) of \(\fg\).  In
  Section~\ref{sec:interpretation-affine-G-action}, we will show that there
  is a $G$-invariant symplectic structure on the affine orbits and
  hence we will be able to rephrase our main result
  (Theorem~\ref{thm:homo-sym-ext}) as a version of
  Corollary~\ref{cor:covering} for $(M,\omega)$ homogeneous symplectic
  $G$-space (but not necessarily hamiltonian) with affine orbits
  replacing the linear coadjoint orbits in the conclusion.
\end{remark}

\subsection{Extending the Lie algebra}
\label{sec:lie-alg-ext}

The cocycle \(c\in\Wedge^2 g^*\) determines a one-dimensional central extension \(\widehat\fg\) of \(\fg\): as a vector space, we have \(\widehat\fg = \fg\oplus\RR\), and the Lie bracket is given by
\begin{equation}
	\comm{(X,u)}{(Y,v)} = (\comm{X}{Y},c(X,Y))
\end{equation}
where \(X,Y\in\fg\), \(u,v\in\RR\). Note that this extension is
trivial if and only if the action of \(G\) on \(M\) is
hamiltonian. Even if it is not hamiltonian, and hence the comoment map
\(\varphi\) is not a Lie algebra homomorphism, it can be
extended to a Lie algebra homomorphism
\(\widehat\varphi:\widehat\fg\to C^\infty(M)\) defined as follows:
\begin{equation}\label{eq:hat-phi}
	\widehat\varphi_{(X,u)}(p) = \varphi_X(p) + u
\end{equation}
for \(p\in M\), \((X,u)\in\widehat\fg\). Indeed, suppressing \(p\) in the notation,
\begin{equation}\label{eq:hat-phi-hom}
	\widehat\varphi_{\comm{(X,u)}{(Y,v)}} = \widehat\varphi_{(\comm{X}{Y},c(X,Y))} = \varphi_{\comm{X}{Y}} + c(X,Y) = \acomm{\varphi_X}{\varphi_Y} = \acomm{\widehat\varphi_{(X,u)}}{\widehat\varphi_{(Y,v)}},
\end{equation}
where the last equality follows because \(u,v\) are constants. Dually, we define an extended moment map \(\widehat\mu:M\to\widehat\fg^*\) by
\begin{equation}\label{eq:hat-mu}
	\pair{\widehat\mu(p)}{(X,u)} = \widehat\varphi_{(X,u)}(p) = \pair{\mu(p)}{X} + u
\end{equation}
where we have used \(\pair{}{}\) to denote the dual pairing on
\(\widehat\fg\) as well as \(\fg\).  Note that
$\widehat\mu(p) = (\mu(p),1) \in \g \oplus \RR$.  Let us define a
linear action of \(G\) on \(\widehat\fg\) by
\begin{equation}\label{eq:hat-g-theta-action}
	g\cdot(X,u) = \qty(\Ad_g X, u-\pair{\theta(g^{-1})}{X})
\end{equation}
where \(\theta\) is the group cocycle appearing in Lemma~\ref{lemma:symp-theta-cocycle}. One can check that the inverse in the argument of \(\theta\) is necessary to make this an action, and it is clearly linear. We have a dual action on \(\widehat\fg^*\cong \fg^*\oplus\RR\) defined by
\begin{equation}
  \pair{g\cdot(\alpha,\zeta)}{(X,u)} = \pair{(\alpha,\zeta)}{g^{-1}\cdot(X,u)} 
\end{equation}
for \((X,u)\in\widehat\fg\), \((\alpha,\zeta)\in\widehat\fg^*\) and \(g\in G\). A simple computation then shows that
\begin{equation}\label{eq:hat-g*-theta-action}
  g\cdot(\alpha,\zeta) = (\Ad^*_g\alpha - \zeta\theta(g),\zeta).
\end{equation}
The extended moment map is equivariant with respect to this action; indeed, we have
\begin{equation}\label{eq:hat-mu-equi}
\begin{split}
  g\cdot \widehat\mu(p) = g\cdot(\mu(p),1) = (\Ad^*_g\mu(p) - \theta(g),1) = (\mu(g\cdot p),1) = \widehat\mu(g\cdot p).
\end{split}
\end{equation}

\begin{remark}
  The action \eqref{eq:hat-g*-theta-action} preserves the hyperplanes
  \(\zeta=\text{constant}\); in particular, the action on the
  \(\zeta=1\) hyperplane is \(g\cdot(\alpha,1)=(\rho(g)\alpha,1)\)
  where \(\rho\) is the affine action defined by
  \eqref{eq:theta-aff-act}. We thus have an equivariant embedding of
  \(\fg^*\) thought of as an affine \(G\)-space into
  \(\widehat\fg^*\). The image of \(\widehat\mu\) lies in this
  hyperplane and corresponds to the image of \(\mu\) in \(\fg^*\).
\end{remark}

One can straightforwardly check that the derivative of the action
\eqref{eq:hat-g-theta-action} is the adjoint representation of
\(\widehat\fg\) restricted to \(\fg\). One might then wonder whether
the central extension \(\widehat\fg\) of \(\fg\) can be ``integrated''
to a central extension \(\widehat{G}\) of \(G\) in such a way
that the actions of \(G\) on \(\widehat\fg\) and its dual described
above arise naturally as restrictions of the adjoint and coadjoint
actions of \(\widehat{G}\). We will show that this is the case.

\section{Symplectic cocycles}
\label{sec:symp-cocycle}

Earlier we saw a moment map for a symplectic action giving rise to a
group cocycle \(\theta\) with values in the coadjoint representation
which could be differentiated and curried to produce a
Chevalley--Eilenberg cocycle \(c\). We will set aside for now the
symplectic action and discuss some properties of such cocycles
abstractly.

\subsection{Differentiating group 1-cocycles}
\label{sec:diff-cocycles}

We begin by deriving a useful identity. Let \(\theta:G\to \fg^*\) be a
group 1-cocycle with values in the coadjoint representation; that is,
a smooth map which satisfies
\begin{equation}\label{eq:theta-cocyle-assump}
  \theta(g_1g_2) = \Ad^*_{g_1}\theta(g_2) + \theta(g_1).
\end{equation}
For \(X\in\fg\), \(g\in G\) and \(t\in \RR\), the cocycle condition
gives us two different expressions for \(\theta(\exp(tX)g)\):
\begin{gather}
  \theta(\exp(tX)g) = \Ad_{\exp(tX)}^*\theta(g) + \theta(\exp(tX)),\\
  \theta(\exp(tX)g) = \theta\qty(g\exp(t\Ad_{g^{-1}}X)) = \Ad_g^*\theta\qty(\exp(t\Ad_{g^{-1}}X)) + \theta(g);
\end{gather}
so for \(Y\in\fg\), we have
\begin{equation}
  \pair{\theta\qty(\exp(t\Ad_{g^{-1}}X))}{\Ad_{g^{-1}}Y} + \pair{\theta(g)}{Y} - \pair{\theta(g)}{\Ad_{\exp(-tX)}Y} - \pair{\theta(\exp(tX))}{Y} = 0,
\end{equation}
and differentiating this with respect to \(t\) at \(t=0\) gives
\begin{equation}\label{eq:theta-comm-ident}
  \pair{\theta(g)}{\comm{X}{Y}} = \pair{d_e\theta(X)}{Y} - \pair{d_e\theta(\Ad_{g^{-1}}X)}{\Ad_{g^{-1}}Y}.
\end{equation}

\begin{lemma}\label{lemma:Dtheta-cocycle}
  If \(\theta\) is a cocycle in \(C^1(G;\fg^*)\) then \(d_e\theta\) is a cocycle in \(C^1(\fg;\fg^*)\).
\end{lemma}

\begin{proof}
  We set \(g=\exp(tZ)\) in equation \eqref{eq:theta-comm-ident} and differentiate with respect to \(t\) at \(t=0\) to find
  \begin{equation}\label{eq:Dtheta-cocycle-calc}
    \begin{split}
      \pair{d_e\theta(Z)}{\comm{X}{Y}} 
      &= - \dv{t}\pair{d_e\theta(\Ad_{\exp(-tZ)}X)}{\Ad_{\exp(-tZ)}Y}\eval_{t=0}\\
      &= - \dv{t}\pair{d_e\theta(\Ad_{\exp(-tZ)}X)}{Y} \eval_{t=0} 
      - \dv{t}\pair{d_e\theta(X)}{\Ad_{\exp(-tZ)}Y}\eval_{t=0}\\
      &= \pair{d_e\theta(\comm{Z}{X})}{Y} + \pair{d_e\theta(X)}{\comm{Z}{Y}}.
    \end{split}
  \end{equation}
  We then abstract \(Y\) and rearrange to find
  \begin{equation}
    \ad_Z^*(d_e\theta(X)) - \ad_X^*(d_e\theta(Z)) - d_e\theta(\comm{Z}{X}) = 0,
  \end{equation}
  hence the claim.
\end{proof}

\subsection{Symplectic cocycles and symplectic group cocycles}

A cochain \(\phi\in C^1(\fg;\fg^*)\) is called \emph{symplectic} if \(\pair{\phi(X)}{Y}=-\pair{\phi(Y)}{X}\). There is a one-to-one correspondence between the space of symplectic cochains \(C_{\mathrm{symp}}^1(\fg;\fg^*)\) and \(C^2(\fg)\) given by \(\phi\mapsto\widetilde\phi\) where \(\widetilde\phi(X,Y):=\pair{\phi(X)}{Y}\). Furthermore, \(\phi\) is a cocycle if and only if \(\widetilde\phi\) is, and all coboundaries in \(C^1(\fg;\fg^*)\) are symplectic, so we can form the symplectic cohomology \(H_{\mathrm{symp}}^1(\fg;\fg^*)\), which is isomorphic to \(H^2(\fg)\). In light of Lemma~\ref{lemma:Dtheta-cocycle}, we say that a cocycle \(\theta\in C^1(G;\fg^*)\) is a \emph {symplectic (group) cocycle} if \(d_e\theta\in C^1(\fg;\fg^*)\) is a symplectic cocycle; we will denote the space of such cocycles by \(Z^1_{\mathrm{symp}}(G;\fg^*)\). We already saw in Lemma~\ref{lemma:Dtheta-symplectic-cocycle} that moment maps for symplectic actions give rise to such cocycles.

\begin{lemma}\label{lemma:symp-cocycles}
	Suppose that \(\theta\) is a symplectic group cocycle. Then the bilinear form \(c\in C^2(\fg)=\Wedge^2 \fg^*\) defined by \(c(X,Y):=\pair{d_e\theta(X)}{Y}\) is the cocycle corresponding to \(d_e\theta\) under the isomorphism \(C^1_{\mathrm{symp}}(\fg;\fg^*)\cong C^2(\fg)\), and furthermore we have
	\begin{equation}
		\pair{\theta(g)}{\comm{X}{Y}} = c(X,Y) - (\Ad_g^*c)(X,Y)
	\end{equation}
	for all \(g\in G\) and \(X,Y\in\fg\), where \((\Ad_g^*c)(X,Y):=c(\Ad^*_{g^{-1}}X,\Ad^*_{g^{-1}}Y)\).
\end{lemma}

\begin{proof}
  The first claim follows from the discussion of symplectic cochains
  and cocycles above since \(c=\widetilde{d_e\theta}\), but for 
  completeness we will show that \(c\) is indeed a cocycle. From the
  calculation~\eqref{eq:Dtheta-cocycle-calc} in the preceding proof, 
  we have
  \begin{equation}\label{eq:c-cocycle-calc}
    c(Z,\comm{X}{Y}) = c(\comm{Z}{X},Y) + c(X,\comm{Z}{Y}),
  \end{equation}
  so, using the skew-symmetry of \(c\),
  \begin{equation}
    c(\comm{X}{Y},Z) + c(\comm{Z}{X},Y) + c(\comm{Y}{Z},X) = 0. 
  \end{equation}
  Thus \(c\) is indeed a cocycle in \(C^2(\fg)\). The second claim
  follows from \eqref{eq:theta-comm-ident}.
\end{proof}

The proofs of the final results of this section use some notational
conventions which may appear odd but are chosen so as to agree with
the notation of Sections~\ref{sec:action-to-extension} and
\ref{sec:group-presymp}. It also uses right-invariant vector fields,
which are introduced in Appendix~\ref{sec:inv-vfs}.

\begin{lemma}\label{lemma:rho-Phi-c-reln}
  Let \(\theta\) be a symplectic group cocycle and \(c\) the
  corresponding Chevalley-Eilenberg cocycle. For each \(X\in\fg\),
  define a function \(\Phi_X\in C^\infty(G)\) by \(\Phi_X(g) =
  -\pair{\theta(g)}{X}\). Then for all \(X,Y\in\fg\) and \(g\in G\),
  \begin{equation}
    (\eL_{\rho_Y}\Phi_X)(g) = (\Ad^*_gc)(X,Y)
  \end{equation}
  where \(\eL\) is the Lie derivative.
\end{lemma}

\begin{proof}
  Using cocycle condition \eqref{eq:theta-cocyle-assump} and
  Lemma~\ref{lemma:symp-cocycles} (or more directly, recalling the
  calculations at the start of Section~\ref{sec:diff-cocycles}),
  \begin{equation}
    (\eL_{\rho_Y}\Phi_X)(g)
    = -\dv{t} \pair{\theta(\exp(tY)g)}{X}\eval_{t=0}
    = (\Ad^*_gc)(X,Y).
  \end{equation}
\end{proof}

\begin{proposition}\label{prop:symp-cocycles-triv}
  Let \(\theta\) be a symplectic group cocycle for a connected
  Lie group \(G\) and \(c\) the induced Chevalley--Eilenberg
  cocycle. Then \([\theta]=0\) if and only if \([c]=0\).
\end{proposition}

\begin{proof}
  First suppose that \(\theta=\partial\alpha\) for some
  \(\alpha\in\fg^*\). Then
  \begin{equation}
    c(X,Y) 
    = \pair{d_e\theta(X)}{Y}
    = \dv{t}\qty[\pair{\alpha}{\Ad_{\exp(-tX)}Y} - \pair{\alpha}{Y}]\eval_{t=0}\\
    =-\pair{\alpha}{\comm{X}{Y}}.
  \end{equation}
  so \(c=\partial_{\mathrm{CE}}\alpha\). Now assume
  \(c=\partial_{\mathrm{CE}}\alpha\) and let
  \(\theta'=\theta-\partial\alpha\). Clearly \(\theta'\) is also a
  symplectic group cocycle in the same cohomology class as
  \(\theta\). We will show that \(\theta'=0\). As in
  Lemma~\ref{lemma:rho-Phi-c-reln}, to each \(X\in\fg\) we assign
  functions \(\Phi_X,\Phi'_X\in C^\infty(G)\) defined by \(\Phi_X(g) =
  -\pair{\theta(g)}{X}\) and \(\Phi'_X(g) = -\pair{\theta'(g)}{X}\);
  then
  \begin{equation}
    \Phi'_X(g) = \Phi_X(g) + \pair{\partial\alpha(g)}{X}.
  \end{equation}
  For all \(X,Y\in\fg\), we thus have
  \begin{equation}
    \eL_{\rho_Y}\Phi_X'
    = \eL_{\rho_Y}\Phi_X + \eL_{\rho_Y}\pair{\partial\alpha}{X}.
  \end{equation}
  where we understand \(\pair{\partial\alpha}{X}\) as the smooth 
  function \(g\mapsto\pair{\partial\alpha(g)}{X}\). On the one hand,
  by Lemma~\ref{lemma:rho-Phi-c-reln} we have
  \((\eL_{\rho_Y}\Phi_X)(g) = (\Ad^*_gc)(X,Y) = 
  	-\pair{\Ad^*\alpha}{\comm{X}{Y}}\),
  while on the other,
  \begin{equation}
    (\eL_{\rho_Y}\pair{\partial\alpha}{X})(g)
    = \dv{t}\pair{\Ad^*_{\exp(tY)g}\alpha}{X}\eval_{t=0}
    = \dv{t}\pair{\Ad^*_{\exp(tY)}\Ad^*_g\alpha}{X}\eval_{t=0}
    = \pair{\Ad^*_g\alpha}{\comm{X}{Y}}.
  \end{equation}
  The assumption on \(c\) therefore gives us
  \(\eL_{\rho_Y}\Phi'_X=0\), so since the values of vector fields
  \(\rho_Y\) span the tangent space at every point in \(G\), we find
  that \(d\Phi'_X=0\). Since \(G\) is connected, this shows that
  \(\Phi'_X\) is a constant for all \(X\in\fg\), so \(\theta'\) is
  constant as a map \(G\to\fg^*\). But \(\theta'\) is a cocycle, so it
  vanishes at the identity and therefore everywhere.
\end{proof}

\begin{proposition}\label{prop:symp-cocycles-unique}
  Let \(\theta\), \(\theta'\) be two symplectic group cocycles for a
  connected group \(G\) which induce the same Chevalley--Eilenberg
  cocycle \(c\). Then \(\theta=\theta'\).
\end{proposition}

\begin{proof}
	By Lemma~\ref{lemma:rho-Phi-c-reln} we have 
	\begin{equation}
		\eL_{\rho_Y}\Phi_X' = \eL_{\rho_Y}\Phi_X = (\Ad^*_gc)(X,Y),
	\end{equation}
	and so \(\eL_{\rho_Y}(\Phi_X'-\Phi_X)=0\). By a similar argument 
	to the previous proof, this shows that the cocycle
	\(\theta'-\theta\) is constant and therefore zero, so 
	\(\theta'=\theta\).
\end{proof}

We finish this section by noting that the preceding pair of results
show that the map \(Z^1_{\mathrm{symp}}(G;\fg^*)\to Z^2(\fg)\) taking
a symplectic group cocycle to the corresponding
Chevalley--Eilenberg cocycle is injective and restricts to an 
isomorphism of the subspaces of coboundaries (note that all
coboundaries in \(Z^1(G;\fg)\) are symplectic), so there is an
induced injective map 
\(H^1_{\mathrm{symp}}(G;\fg^*)\hookrightarrow H^2(\fg)\).

\section{One-dimensional central extensions of Lie groups}
\label{sec:group-ext}

In Section~\ref{sec:lie-alg-ext}, we used data obtained from a
symplectic action of a Lie group \(G\) to construct a central
extension \(\widehat\fg\) of the Lie algebra of \(G\) as well as
actions of \(G\) on \(\widehat\fg\) and on its dual \(\widehat\fg^*\)
and asked whether such actions of $G$ might result from the adjoint
and coadjoint representations of a central extension of \(G\) itself.  We
now consider the existence of such group extensions and their adjoint
and coadjoint representations. We start with a definition.

\begin{definition}
  Let \(G\) be a connected Lie group. A \emph{one-dimensional central
    extension} of \(G\) is a short exact sequence of Lie groups
  \begin{equation}\label{eq:grp-cent-ext}
  \begin{tikzcd}
   1 \arrow[r] & K \arrow[r] & \widehat{G} \arrow[r] &
   G \arrow[r] & 1.
  \end{tikzcd}
  \end{equation}
  where \(K\) is a one-dimensional abelian group whose image lies in
  the centre of \(\widehat{G}\). If we do not wish to specify \(K\),
  we may say that \(\widehat{G}\) is a one-dimensional central extension
  of \(G\).
\end{definition}

\subsection{Adjoint and coadjoint representations}
\label{sec:adj-coadj}

Let $\widehat G$ be a one-dimensional central extension of $G$.  We
denote by \(\Ad: G\to\GL(\fg)\) and \(\widehat\Ad:
\widehat{G}\to\GL(\widehat\fg)\) the adjoint representations of \(G\)
and \(\widehat{G}\) respectively; the coadjoint representations are
denoted \(\Ad^*: G\to\GL(\fg^*)\) and \(\widehat\Ad^*:
\widehat{G}\to\GL(\widehat\fg^*)\). Let us identify \(K\) with its
image in \(\widehat{G}\) and regard \(G\) as \(G=\widehat{G}/K\).
Since \(K\) is central, its coadjoint action on \(\widehat\fg\) is
trivial so \(\widehat\Ad\) factors through \(G\) to give a
representation of \(G\) on \(\widehat\fg\), also denoted
\(\widehat\Ad\), such that \(\widehat\Ad_g=\widehat\Ad_{\widehat{g}}\)
where \(\widehat{g}\) is any lift of \(g\in G\); that is, there exists
a dotted arrow making the following diagram commute:
\begin{equation}
  \begin{tikzcd}
    \widehat{G} \ar[rr, "\widehat\Ad"] \ar[dr,two heads] 
    & & \GL(\widehat\fg)\\
    & G=\widehat{G}/K \ar [ur,dashed,"\widehat\Ad" below]&
  \end{tikzcd}.
\end{equation}
Of course, \(K\) also acts trivially via the coadjoint action which
thus factors through \(G\) similarly to give a representation
\(\widehat\Ad^*:G\to\GL(\widehat\fg^*)\) dual to
\(\widehat\Ad:G\to\GL(\widehat\fg)\). On the other hand, the adjoint
and coadjoint representations of \(G\) give rise to representations of
\(\widehat{G}\) by pulling back along the quotient map:
\begin{equation}
  \begin{tikzcd}
    \widehat{G} \ar[rr,dashed,"\Ad"] \ar[dr,two heads] 
    & & \GL(\fg)\\
    & G=\widehat{G}/K \ar [ur,"\Ad" below]&
  \end{tikzcd}.
\end{equation}
One can check that the canonical map
\(\widehat{\fg}\twoheadrightarrow\fg\) and the dual embedding
\(\fg^*\hookrightarrow\widehat{\fg}^*\) are \(G\)-equivariant
(equivalently \(\widehat{G}\)-equivariant).  We shall refer to
$\widehat\Ad$ and $\widehat\Ad^*$ as the \emph{factored} adjoint and
coadjoint actions of $G$ on $\widehat\g$ and $\widehat\g^*$,
respectively.

We now pick a complement to \(\Lie(K)=\fk\cong\RR\) in
\(\widehat{\fg}\) and identify it with \(\fg\) so that we have a (not
necessarily invariant) splitting of vector spaces
\(\widehat{\fg}=\fg\oplus\RR\). Since \(K\) is central,
\(\widehat\Ad\) acts trivially on the \(\RR\) factor. Using this and
equivariance of the quotient map, we conclude that the hatted adjoint
action is of the form
\begin{equation}\label{eq:hat-adjoint-theta}
  \widehat\Ad_g(X,u) := \qty(\Ad_g X, u - \pair{\theta(g^{-1})}{X})
\end{equation}
for some map \(\theta:G\to\fg^*\), where the reasons for the sign and
the inversion in the argument will become clear soon.

The splitting \(\widehat\fg=\fg\oplus\RR\) induces a dual splitting
\(\widehat\fg^*=\fg^*\oplus\RR\) such that the natural embedding
\(\fg^*\hookrightarrow\widehat{\fg}^*\) is given by
\(\alpha\mapsto(\alpha,0)\). The natural paring between
\(\widehat{\fg}\) and \(\widehat{\fg}^*\) is then
\begin{equation}
	\pair{(\alpha,\zeta)}{(X,u)} = \pair{\alpha}{X} + \zeta u
\end{equation}
where \(\alpha\in\fg^*\), \(X\in\fg\) and \(\zeta,u\in\RR\). The factored coadjoint action can be explicitly computed as follows:
\begin{equation}
\begin{split}
	\pair{\widehat\Ad^*_g(\alpha,\zeta)}{(X,u)} 
	&:= \pair{(\alpha,\zeta)}{\widehat\Ad_{g^{-1}}(X,u)}\\
	&= \pair{(\alpha,\zeta)}{\qty(\Ad_{g^{-1}} X, u - \pair{\theta(g)}{X})}\\
	&= \pair{\alpha}{\Ad_{g^{-1}}X} + \zeta u - \zeta\pair{\theta(g)}{X}\\
	&= \pair{\Ad^*_g\alpha - \zeta\theta(g)}{X} + \zeta u.
\end{split}
\end{equation}
Abstracting away \((X,u)\), we thus have
\begin{equation}\label{eq:hat-coadjoint-theta}
	\widehat\Ad^*_g(\alpha,\zeta) = (\Ad^*_g\alpha - \zeta\theta(g),\zeta).
\end{equation}

\begin{lemma}\label{lemma:theta-cocycle}
  The map \(\theta:G\to\fg^*\) in
  equation~(\ref{eq:hat-coadjoint-theta}) is a symplectic group
  cocycle, and the associated Chevalley--Eilenberg cocycle
  \(c\in\Wedge^2\fg^*\) (see Lemma~\ref{lemma:symp-cocycles}) is the
  one determining the presentation of \(\widehat\fg\) as an extension
  of \(\fg\).  Furthermore, their cohomology classes do not depend on
  the splitting of \(\widehat{g}\).
\end{lemma}
	
\begin{proof}
  Let \(g_1,g_2\in G\). Since \(\widehat\Ad^*\) is a
  representation, \(\widehat\Ad^*_{g_1g_2}=\widehat\Ad^*_{g_1}
  \widehat\Ad^*_{g_2}\) so we see that 
  \begin{equation}
    \begin{split}
      \qty(\Ad^*_{g_1g_2}\alpha - \zeta\theta(g_1g_2),\zeta) &=
      \widehat\Ad^*_{g_1}\qty(\Ad^*_{g_2}\alpha - \zeta\theta(g_2),\zeta)\\
      &= \qty(\qty(\Ad^*_{g_1}\Ad^*_{g_2}\alpha - \zeta\Ad^*_{g_1}\theta(g_2)) - \zeta\theta(g_1),\zeta)
    \end{split}
  \end{equation}
  for all \(\alpha\in\fg^*\) and \(\zeta\in\RR\).  Then we use that
  \(\Ad^*\) is a representation and abstract \(\zeta\) to find
  \begin{equation}\label{eq:theta-cocycle}
    \theta(g_1g_2) = \Ad^*_{g_1}\theta(g_2) + \theta(g_1),
  \end{equation}
  thus \(\theta\) is a cocycle. We now differentiate \eqref{eq:hat-adjoint-theta} to find
  \begin{equation}
    \comm{(X,u)}{(Y,v)}
    = \qty(\comm{X}{Y},c(X,Y))
  \end{equation}
  where \(c(X,Y):=\pair{d_e\theta(X)}{Y}\); in particular, \(\theta\)
  is symplectic as required. If a different splitting of
  \(\widehat\fg\) is chosen then \([c]=[c']\), since the cohomology
  class of \(c\) is determined by the isomorphism class of the
  extension. Now note that \(\theta'-\theta\) is a symplectic group
  cocycle and its cohomology class is trivial by
  Proposition~\ref{prop:symp-cocycles-triv} since \([c'-c]=0\), so
  \([\theta']=[\theta]\).
\end{proof}

\begin{remark}
	Note that the action \eqref{eq:hat-coadjoint-theta} preserves the affine hyperplanes \(\zeta=\mathrm{constant}\), so in particular if we identify \(\fg^*\) with the affine hyperplane \(\zeta=1\) in \(\widehat\fg^*\), we obtain an affine action \(\rho\) of \(G\) on \(\fg^*\) given by
	\begin{equation}\label{eq:aff-act-dual}
		\rho(g)\alpha = \Ad^*_g\alpha - \theta(g).
	\end{equation}
	This is the same formula as \eqref{eq:theta-aff-act}, except there the cocycle \(\theta\) arose from a symplectic action rather than from the coadjoint action of a central group extension. The actions of \(G\) on \(\widehat\fg\) and \(\widehat\fg^*\) defined in Section~\ref{sec:lie-alg-ext} are also of the same form as the hatted adjoint and coadjoint representations here, which again prompts the question of whether those actions do indeed arise from a central group extension. The next section contains the key result we need to answer this question.
\end{remark}

\subsection{Existence of extensions}

We now present a necessary and sufficient criterion for the existence
of central extensions of Lie groups due to Neeb \cite{MR1424633}. We
use the notation for left- and right-translations and left- and
right-invariant vector fields on \(G\) introduced in
Appendix~\ref{sec:inv-vfs}.

Let \(c\in \Wedge^2\fg^*\) be a Chevalley--Eilenberg 2-cocycle for
\(\fg\). We showed in Section~\ref{sec:lie-alg-ext} that this defines
a one-dimensional central extension of \(\fg\), and recalling the
isomorphism between the Chevalley--Eilenberg complex for \(\fg\) and
the left-invariant de Rham complex on \(G\) discussed in
Appendix~\ref{sec:left-inv-forms}, there exists a closed,
left-invariant 2-form on \(G\) defined by \(\Omega_g =
L_{g^{-1}}^*c\), or equivalently,
\begin{equation}\label{eq:Omega-def}
	\Omega(\lambda_X,\lambda_Y)(g) = c(X,Y)
\end{equation}
for all \(g\in G\). We may now state Neeb's result.

\begin{theorem}[Neeb \cite{MR1424633}]\label{thm:neeb}
  Let \(G\) be a connected Lie group, \(c\in\Wedge^2\fg^*\) a
  Chevalley--Eilenberg 2-cocycle of the Lie algebra \(\fg\) defining a
  one-dimensional central extension
  \(0\to\RR\to\widehat\fg\to\fg\to0\) and \(\Omega\in\Omega^2(G)\) the
  corresponding left-invariant closed 2-form. Then there exists a
  one-dimensional central extension \(\widehat{G}\) of \(G\)
  integrating the above central extension if and only if for any
  \(X\in\fg\) there exists a function \(\Phi_X\in C^\infty(G)\)
  satisfying
  \begin{equation}\label{eq:neeb}
    \imath_{\rho_X}\Omega = d\Phi_X.
  \end{equation}
\end{theorem}

One might observe that equation \eqref{eq:neeb} is similar to
\eqref{eq:symp-ham-vfs}; indeed, if we interpret \(\Omega\) as a
presymplectic structure on \(G\), this equation says that
\(X\mapsto\Phi_X\) is a comoment map for the left action of \(G\)
on itself. Indeed, although Neeb does not interpret the equation
in this way in \cite{MR1424633}, the earlier work of Tuynman and
Wiegerinck \cite{MR948561} upon which Neeb builds does. We will 
return to this point of view in Section~\ref{sec:group-presymp}.

\begin{remark}
	Neeb's result is actually slightly stronger than the
	above, although we will not need the stronger version; fixing a
	connected one-dimensional Lie group \(K\), functions \(\Phi_X\)
	satisfying \eqref{eq:neeb} exist if and only if there is a central
	extension \(\widehat{G}\) of \(G\) by \(K\) integrating the extension
	of Lie algebras. Note, however, that the theorem above does not say 
	that there is a \emph{unique} such extension, only that one exists
	-- even for a fixed choice of kernel group \(K\).
\end{remark}

\section{Symplectic structures and central extensions}
\label{sec:sympl-struct-central-extn}

Finally we bring the preceding discussion together and present our
main result. We will again use the results and notation of
Appendix~\ref{sec:lie-group-structures} freely.

\subsection{From the symplectic action to the central extension}
\label{sec:action-to-extension}

Let us recapitulate some of the main points of the preceding discussion.
In Section~\ref{sec:symp-action}, we considered a symplectic homogeneous
\(G\)-space \(M\), with \(G\) a connected Lie group, with a moment map
\(\mu:M\to\fg^*\) and comoment map \(\varphi:\fg\to C^\infty(M)\). 
We now assume that \(M\) is simply-connected, so such maps exist, but we
do not assume that the action is hamiltonian. 
By Lemma~\ref{lemma:symp-theta-cocycle}, the failure of equivariance of
the moment map is measured by a group cocycle \(\theta:G \to \fg^*\)
which is symplectic by Lemma~\ref{lemma:Dtheta-symplectic-cocycle} and gives
rise to a Chevalley--Eilenberg 2-cocycle \(c\in\wedge^2 \fg^*\) and
thus to a closed, left-invariant 2-form \(\Omega\) on \(G\) given by
equation~\eqref{eq:Omega-def}. In Section~\ref{sec:lie-alg-ext}, we
used \(c\) to construct a central extension \(\widehat\fg\) of \(\fg\),
upon which \(G\) acts via an action \eqref{eq:hat-g-theta-action}
determined by \(\theta\). There is a dual action
\eqref{eq:hat-g*-theta-action} of \(G\) on \(\widehat\fg^*\) with
respect to which the extended moment map
\(\widehat\mu:M\to\widehat\fg^*\) defined by equation \eqref{eq:hat-mu}
is equivariant. Dual to this, we have an extended comoment map 
\(\widehat\varphi: \widehat\fg\to C^\infty(M)\) defined by equation 
\eqref{eq:hat-phi} which is a Lie algebra homomorphism.
One is naturally lead to ask whether there exists a
central group extension \(\widehat{G}\) of \(G\) integrating the algebra
extension, and furthermore whether the actions of \(G\) on
\(\widehat\fg\) and \(\widehat\fg^*\) arise from the adjoint and
coadjoint representations of \(\widehat{G}\), and finally whether
\(\widehat\mu\) is a moment map for an action of \(\widehat{G}\) on
\(M\). Theorem~\ref{thm:neeb} allows us to finally answer all of these
questions in the affirmative.

\begin{theorem}\label{thm:homo-sym-ext}
  Let $G$ be a connected Lie group and \((M,\omega)\) be a
  simply-connected symplectic homogeneous manifold of \(G\).  Then
  there exists a one-dimensional central extension \(\widehat{G}\) of
  \(G\) such that \(M\) is the universal cover of a coadjoint orbit of
  \(\widehat{G}\).
\end{theorem}

\begin{proof}
  Recall the functions \(\Phi_X\in C^\infty(G)\) given by
  \(\Phi_X(g)=-\pair{\theta(g)}{X}\) which we first saw in
  Lemma~\ref{lemma:rho-Phi-c-reln}. We will show that these satisfy 
  the condition~\eqref{eq:neeb} in Theorem~\ref{thm:neeb}. Recall
  that the right-invariant vector fields span \(\eX(G)\) as a
  \(C^\infty(G)\)-module, so it is sufficient to show that
  \begin{equation}
    \Omega(\rho_X,\rho_Y) = \eL_{\rho_Y}\Phi_X.
  \end{equation}
  By Lemma~\ref{lemma:rho-Phi-c-reln}, we have
  \(\eL_{\rho_Y}\Phi_X(g)=(\Ad^*_g c)(X,Y)\); on the other hand,
  since \(\left( \rho_X \right)_g=\left(\lambda_{\Ad_{g^{-1}}
      X}\right)_g\) at \(g\in G\),
  \begin{equation}
    \Omega(\rho_X,\rho_Y)(g)
    = \Omega(\lambda_{\Ad_{g^{-1}}X},\lambda_{\Ad_{g^{-1}}Y})(g)
    = c(\Ad_{g^{-1}}X,\Ad_{g^{-1}}Y)
    = (\Ad^*_g c)(X,Y).
  \end{equation}
  So we indeed have \eqref{eq:neeb}, and by Theorem~\ref{thm:neeb},
  there exists a one-dimensional central extension \(\widehat{G}\) of
  \(G\) integrating the central extension \(\widehat\fg\) of \(\fg\)
  defined by \(c\).  The action of \(G\) on \(M\) pulls back to
  \(\widehat{G}\) along the quotient map \(\pi:\widehat{G}\to G\); for
  \(g\in\widehat{G}\) and \(p\in M\), we have \(g\cdot p=\pi(g)\cdot p\), and
  clearly this action is symplectic and transitive.
  
  Let us write \(\xi_{(X,u)}\) for the fundamental vector field on \(M\) 
  of the \(\widehat{G}\)-action generated by 
  \((X,u)\in\widehat\fg=\fg\oplus\RR\). But then, since \((0,u)\in\ker\pi_*\),
  we have \(\xi_{(X,u)}=\xi_X\), where the latter is 
  the fundamental vector field for the \(G\)-action as before. Now recall the
  extended comoment map \(\widehat\varphi:\widehat\fg\to M\) given by
  \(\widehat\varphi_{(X,u)}=\varphi_X+u\). We have
  \begin{equation}
    d\widehat\varphi_{(X,u)} = d\varphi_X = \imath_{\xi_X}\omega = \imath_{\xi_{(X,u)}}\omega,
  \end{equation}
  so \(\widehat\varphi\) is indeed a comoment map. We already showed 
  (see \eqref{eq:hat-phi-hom}) that it is a Lie algebra homomorphism,
  so (since \(\widehat{G}\) is connected) we conclude that 
  \((M,\omega,\widehat\mu)\) is a hamiltonian \(\widehat{G}\)-space.
  Corollary~\ref{cor:covering}, with \(\widehat{G}\) now playing the
  rôle of \(G\), then shows that \(\widehat\mu: M \to \widehat\bigO\)
  is a covering of a coadjoint orbit \(\widehat\bigO\) and, since \(M\)
  is simply-connected, it is the universal cover.
\end{proof}

\subsection{Interpretation in terms of \(G\)-actions}
\label{sec:interpretation-G-action}

In the preceding proof, after proving the existence of the central
extension \(\widehat{G}\) of \(G\), we passed from the action of the
original group \(G\) to the action of \(\widehat{G}\), which we showed
was hamiltonian, and thus we were able to use results about hamiltonian
\(\widehat{G}\)-spaces for a concise proof. It is possible, however, 
to instead work entirely with \(G\)-spaces by factoring the adjoint and
coadjoint actions of \(\widehat{G}\) through \(G\) as in 
Section~\ref{sec:adj-coadj}, and it is instructive to see how the
situation looks from this point of view, especially since it makes
the rôle of the symplectic cocycle \(\theta\) more explicit.

Recall that, by construction, \(\widehat\fg=\fg\oplus\RR\) as vector
spaces with \(c\) as the Chevalley--Eilenberg cocycle associated to
the splitting, and by equation~\eqref{eq:hat-coadjoint-theta} the 
factored coadjoint action is given by
\begin{equation}
  (\alpha,\zeta) \mapsto (\Ad^*_g\alpha - \zeta\theta(g),\zeta).
\end{equation}
The symplectic cocycle \(\theta\) appearing here is not a priori the
same as the one arising from the moment map \(\mu\), but since since
both integrate \(c\), they are in fact the same by
Proposition~\ref{prop:symp-cocycles-unique}. This action is thus
identical to the one defined by
equation~\eqref{eq:hat-g*-theta-action}, so the extended moment map
\(\widehat\mu\) is \(G\)-equivariant with respect to \(\widehat\Ad^*\)
by the calculation \eqref{eq:hat-mu-equi}. Similarly, the action
defined by \eqref{eq:hat-g-theta-action} is just the factored adjoint
action. This answers the questions posed at the beginning of
Section~\ref{sec:action-to-extension} explicitly.

Let us now recall the proof of
Proposition~\ref{prop:moment-map-is-morphism} and borrow some of its
notation, albeit with some differences since \((M,\omega,\mu)\) is not
a hamiltonian \(G\)-space here. We fix an arbitrary point \(o\in M\)
and let \(\lambda=\widehat\mu(o)=(\mu(o),1)\). Since the action of
\(G\) is transitive on \(M\), any point in \(M\) can be expressed as
\(g\cdot o\) for some \(g\in G\), and we have
\begin{equation}
  \widehat\mu(g\cdot o) = \widehat\Ad^*_g\widehat\mu(o) =
  \widehat\Ad^*_{g}\lambda,
\end{equation}
so in particular the image of \(\widehat\mu\), denoted by
\(\widehat\bigO\) above, is the orbit of \(\lambda\) under the
factored coadjoint action. Let \(a_o:G\to M\) and
\(a_\lambda:G\to\widehat\bigO\) be the orbit maps \(g\mapsto g\cdot o\)
and \(g\mapsto\widehat\Ad^*_g\lambda\) respectively. The
\(G\)-equivariance of \(\widehat\mu\) means that the following
triangle of \(G\)-equivariant maps commutes:
\begin{equation}
  \label{eq:M-G-O-mu-triangle-too}
  \begin{tikzcd}
    & G \arrow[dl,"a_o"'] \arrow[dr,"a_\lambda"] &\\
    M \arrow[rr,"\widehat\mu"] & & \widehat\bigO.
  \end{tikzcd}
\end{equation}
Let us denote the fundamental vector fields of the \(G\)-action on
\(\widehat\bigO\) by \(\zeta_X\). We have \(\left(\xi_X\right)_o = (a_o)_* X\) and
\(\left(\zeta_X\right)_\lambda = (a_\lambda)_* X\) for each \(X\in\fg\), and these
values span the tangent spaces of \(M\) and \(\widehat\bigO\) respectively. 
Similarly to the computations in the proof of
Proposition~\ref{prop:moment-map-is-morphism}, we can compute the values at
\(o\) of the pullback of the Kirillov--Kostant--Souriau structure
\(\widehat\omega_{\mathrm{KKS}}\) on \(\widehat\bigO\) and of the symplectic form \(\omega\)
on \(M\):
\begin{align}
  \label{eq:KKS-bigOhat}
  (\widehat\mu^*\widehat\omega_{\mathrm{KKS}})(\xi_X,\xi_Y)(o)
  &=\widehat\omega_{\mathrm{KKS}}(\zeta_{X},\zeta_{Y})(\lambda)
    = \pair\lambda{\comm{(X,0)}{(Y,0)}}
    = \pair{\mu(o)}{\comm{X}{Y}} + c(X,Y),\\
  \omega(\xi_X,\xi_Y)(o)
  \label{eq:s}
  &= \acomm{\varphi_X}{\varphi_Y}(o)
    = \varphi_{\comm{X}{Y}}(o) + c(X,Y)
    = \pair{\mu(o)}{\comm{X}{Y}} + c(X,Y),
\end{align}
which shows that \(\widehat\mu:M\to\widehat\bigO\) is a symplectomorphism
(as usual, it is sufficient to compute at one point by \(G\)-invariance and
transitivity), and thus a homomorphism of homogeneous symplectic
\(G\)-spaces.

Since \(\widehat\mu^*\widehat\omega_{\mathrm{KKS}}=\omega\) on \(M\), commutativity
of the diagram \eqref{eq:M-G-O-mu-triangle-too} implies that we also have
\(a_o^*\omega=a_\lambda^*\widehat\omega_{\mathrm{KKS}}\) on \(G\).
One might ask whether the orbit maps are also symplectic in the sense 
that this invariant 2-form is also equal to \(\Omega\). By \(G\)-invariance 
and since \(G\) acts transitively on itself, it suffices to compute at the
identity. We have
\begin{equation}
\begin{split}
	(\Omega-a_o^*\omega)(X,Y)(e)
	&= c(X,Y) - \omega(\xi_X,\xi_Y)(o)\\
	&= c(X,Y) - \pair{\mu(o)}{\comm{X}{Y}} - c(X,Y) \\
	&= -\pair{\mu(o)}{\comm{X}{Y}}.
\end{split}
\end{equation}
Thus \((\Omega-a_o^*\omega)_e=\partial_{\mathrm{CE}}\mu(o)\),
so the orbit maps are not symplectic on the nose but we have
\([\Omega]=[a_o^*\omega]\) in de Rham cohomology. However, since we have
freedom to redefine \(\mu\) by addition of an element in \(\fg^*\), we
can without loss of generality assume that \(\mu(o)=0\), and then the
orbit maps are indeed symplectic. Note that with this assumption,
we also have \(\lambda=(0,1)\in\widehat\fg^*\); this is the element
denoted \(\lambda\) in \cite{MR1424633}.

\subsection{Interpretation in terms of the affine \(G\)-action}
\label{sec:interpretation-affine-G-action}

Recall the affine action \(\rho:G\mapsto\Aff(\fg^*)\) given by 
\(\rho(g)\alpha=\Ad^*_g\alpha - \theta(g)\) 
(equations \eqref{eq:theta-aff-act} and \eqref{eq:aff-act-dual}). We noted
in previous remarks that there is a \(G\)-equivariant embedding of \(\fg^*\)
in \(\widehat\fg^*=\fg^*\oplus\RR\) as the hyperplane \((\fg,1)\), with the
\(\rho\) action on \(\fg\) and the factored coadjoint action on
\(\widehat\fg^*\); \(\widehat\Ad^*_g(\alpha,1) = (\rho(g)\alpha,1)\).
We now interpret Theorem~\ref{thm:homo-sym-ext} in terms of this action.

Let \(\alpha\in\fg\) and let \(\bigO^{\text{aff}}_\alpha\) be the
orbit of \(\alpha\in\fg\) under the affine action \(\rho\). Completely
analogous to the KKS structure on a coadjoint orbit, there is a
natural invariant symplectic structure on this orbit. Let
\(G^{\text{aff}}_\alpha\) be the stabiliser of \(\alpha\) under
\(\rho\) and let \(\fg^{\text{aff}}_\alpha\) be its Lie algebra, so
\begin{equation}\label{eq:aff-stab-group}
  G^{\text{aff}}_\alpha = \left\{g\in G ~ \middle |~ \Ad^*_g\alpha - \theta(g) = \alpha \right\}
\end{equation}
and
\begin{equation}\label{eq:aff-stab-alg}
  \fg^{\text{aff}}_\alpha = \left\{X\in \fg~ \middle |~ \pair{\alpha}{\comm{X}{Y}} + c(X,Y) = 0\quad \forall Y\in\fg \right\}.
\end{equation}
Since \(G^{\text{aff}}_\alpha\) fixes \(\alpha\), the pushforward by the 
action of an element \(g\in G^{\text{aff}}_\alpha\) preserves 
\(T_\alpha \bigO^{\text{aff}}_\alpha\), hence we obtain an action of 
\(G^{\text{aff}}_\alpha\) on  \(T_\alpha \bigO^{\text{aff}}_\alpha\).
Denoting by \(\zeta_X\) the fundamental vector field on 
\(\bigO^{\text{aff}}_\alpha\) associated to \(X\in\fg\), there is a linear
map \(\fg\to T_\alpha\bigO^{\text{aff}}_\alpha\) given by
\(X\mapsto(\zeta_X)_\alpha\), the kernel of which is \(\fg^{\text{aff}}_\alpha\).
Since \(\rho(g)_*\zeta_X=\zeta_{\Ad_gX}\) for all \(g\in G\), in particular
this map is \(G^{\text{aff}}_\alpha\)-equivariant, there is a short exact
sequence of \(G^{\text{aff}}_\alpha\)-modules
\begin{equation}
  \begin{tikzcd}
    0 \arrow[r] & \fg^{\text{aff}}_\alpha \arrow[r] & \fg \arrow[r] &
    T_\alpha\bigO^{\text{aff}}_\alpha \arrow[r] & 0.
  \end{tikzcd}
\end{equation}
There exists a \(G^{\text{aff}}_\alpha\)-invariant 2-form on \(\fg\) given
by \((X,Y)\mapsto \pair{\alpha}{\comm{X}{Y}} + c(X,Y)\) which by
\eqref{eq:aff-stab-alg} descends to a nondegenerate,
\(G^{\text{aff}}_\alpha\)-invariant 2-form on 
\(T_\alpha\bigO^{\text{aff}}_\alpha\cong \fg/\fg^{\text{aff}}_\alpha\),
and by the holonomy principle this in turn extends to a nondegenerate,
\(G\)-invariant 2-form \(\omega_{\mathrm{aff}}\) on \(\bigO^{\text{aff}}_\alpha\).
One can then verify that \(d\omega_{\mathrm{aff}}=0\), so \(\omega_{\mathrm{aff}}\)
is a symplectic form. One can also check that the inclusion
\(i:\bigO_\alpha^\text{aff}\hookrightarrow\fg^*\) is a moment map, although
unlike in the case of coadjoint orbits, it is not equivariant (with the 
coadjoint action on \(\fg^*\)), so
\((\bigO_\alpha^\text{aff},\omega_{\mathrm{aff}},i)\) is not a hamiltonian
\(G\)-space.

The embedding \(\fg^*\hookrightarrow\widehat\fg^*=\fg^*\oplus\RR\) restricts
to an isomorphism of homogeneous \(G\)-spaces
\(\bigO^{\text{aff}}_\alpha\to\widehat\bigO_{(\alpha,1)}\), where the latter
space is the coadjoint orbit of \((\alpha,1)\). We can compute
\begin{equation}
	(\widehat\omega_{\mathrm{KKS}})(\zeta_X,\zeta_Y)(\alpha,1)
		= \pair{(\alpha,1)}{\comm{(X,0)}{(Y,0)}}
		= \pair{(\alpha,1)}{(\comm{X}{Y},c(X,Y))}
		= \pair{\alpha}{\comm{X}{Y}} + c(X,Y)
\end{equation}
so the isomorphism is also a symplectomorphism.

Setting \(\alpha=\mu(o)\), we have \((\alpha,1)=\lambda=\widehat\mu(o)\), 
so we see that the affine orbit
\(\bigO^{\text{aff}}_\alpha\) is isomorphic to the coadjoint orbit of
\(\lambda\) as a homogeneous symplectic \(G\)-space. We can thus
interpret Theorem~\ref{thm:homo-sym-ext} as saying that \(M\) covers
an \emph{affine} orbit of \(G\) as a homogeneous symplectic
\(G\)-space.

\subsection{Invariant presymplectic structures on groups, symplectic cocycles and central extensions}
\label{sec:group-presymp}

The lynchpin in the proof of Theorem~\ref{thm:homo-sym-ext} is the
symplectic group cocycle \(\theta\) arising from the moment map for
the action of \(G\) on \(M\). We also saw in
Section~\ref{sec:adj-coadj} that factoring the coadjoint action of a
central extension of \(G\) through \(G\) itself gives rise to a
symplectic group cocycle (after splitting the central extension of
algebras), and that this was the same cocycle when the central
extension was defined by \(c=\widetilde{d_e\theta}\). Forgetting the
action of \(G\) on \(M\), what we have implicitly shown is the
following.

\begin{theorem}\label{thm:group-ext-symp-cocycle}
  Let \(G\) be a connected Lie group and
  \begin{equation}\label{eq:lie-alg-ext}
    \begin{tikzcd}
      0 \arrow[r] & \RR \arrow[r] & \widehat{\fg} \arrow[r] &
      \fg \arrow[r] & 0
    \end{tikzcd}
  \end{equation}
  a central extension of the Lie algebra \(\fg\) of \(G\). This
  integrates to a one-dimensional central extension of the group \(G\)
  \begin{equation}\label{eq:lie-group-ext}
    \begin{tikzcd}
      1 \arrow[r] & K \arrow[r] & \widehat{G} \arrow[r] &
      G \arrow[r] & 1
    \end{tikzcd}
  \end{equation}
  if and only if there exists a symplectic group cocycle
  \(\theta\in C^1(G;\fg^*)\) for which the cohomology class of
  the Chevalley--Eilenberg 2-cocycle \(c\) of \(\fg\) given by
  \(c(X,Y)=\pair{d_e\theta(X)}{Y}\) corresponds to the extension
  \eqref{eq:lie-alg-ext}.
\end{theorem}

We saw at the end of Section~\ref{sec:symp-cocycle} that the map
assigning a Chevalley--Eilenberg cocycle \(c\) to a symplectic group
cocycle \(\theta\) induces an injection
\(H^1_{\mathrm{symp}}(G,\fg^*)\hookrightarrow H^2(\fg)\). Recalling
that \(H^2(\fg)\) classifies the one-dimensional central extensions
of \(G\), we now see that \(H^1_{\mathrm{symp}}(G,\fg^*)\) (or more
precisely its image in \(H^2(\fg)\)) classifies the \emph{integrable}
extensions.

We will now shift to a slightly different perspective which will serve
to unify a few different themes already present in this work and prove
the above theorem more concisely. Recall that, given a Lie algebra
extension \eqref{eq:lie-alg-ext}, choosing a splitting (as vector
spaces) yields a Chevalley--Eilenberg cocycle \(c\in\Wedge^2 \fg\) and
a closed, left-invariant two-form \(\Omega\) on \(G\) corresponding to
\(c\). A different choice of splitting gives rise a cocycle and 2-form
in the same cohomology classes \([c]\) and \([\Omega]\), respectively.

We now return to the observation following Theorem~\ref{thm:neeb} and
emphasise that \(\Omega\) is a left-invariant closed 2-form on $G$;
that is, a left-invariant presymplectic structure, making \(G\) into a
presymplectic homogeneous space for itself. The fundamental vector
fields of this action are the \emph{right}-invariant vector fields
\(\rho_X\). With this point of view, let us compare the equations
\eqref{eq:neeb} and \eqref{eq:symp-ham-vfs} and note that we can
perform the same analysis on \((G,\Omega)\) as we did on
\((M,\omega)\), \emph{mutatis mutandis}, the only major difference
being that there is no Poisson bracket on \(G\) in general since
\(\Omega\) is possibly degenerate. That is, where they exist, we can
choose the \(\Phi_X\) such that the map \(\Phi: \fg\to C^\infty(G)\)
is linear and consider it as a comoment map with corresponding moment
map \(\mu:G\to\fg^*\), both unique up to addition of an element in
\(\fg^*\). Continuing with the analogy, we discover that (as in
Lemmas~\ref{lemma:symp-theta-cocycle} and
\ref{lemma:Dtheta-symplectic-cocycle}) there exists a symplectic group
cocycle \(\theta:G\to\fg^*\) satisfying
\begin{equation}
	\theta(g_1) = \Ad_{g_1}^*\mu(g_2) - \mu(g_1g_2)
\end{equation}
for all \(g_1, g_2\in G\) (in particular the RHS does not depend on
\(g_2\)). Note, however, that since \(\mu\) is defined only up to an
element of \(\fg^*\), we may assume without loss of generality that
\(\mu(e)=0\), in which case we find
\begin{equation}
	\theta(g) = \Ad_{g}^*\mu(e) - \mu(g) = -\mu(g);
\end{equation}
in particular, the moment map \(\mu\) is itself a cocycle and we have
\(\Phi_X(g)=-\pair{\theta(g)}{X}\). It is simple to check now that the
Chevalley--Eilenberg cocycle obtained by taking the derivative of
\(\theta\) is the cocycle \(c\);
\begin{equation}
  \widetilde{d_e\theta}(X,Y) := \pair{d_e\theta(X)}{Y}
  = -\dv{t} \Phi_Y(\exp(tX)) \eval_{t=0}\\
  = -(\eL_{\rho_X} \Phi_Y)(e)
  = \Omega(\rho_X,\rho_Y)(e)
  = c(X,Y).
\end{equation}
Notice that if we do not assume that \(\mu(e)=0\), the claim is true
up to a coboundary, i.e., \(\qty[\widetilde{d_e\theta}]=[c]\).

Recalling Proposition~\ref{prop:hamiltonian-action}, one might ask
what it means for the action of \(G\) on itself to be ``hamiltonian''
with respect to \(\Omega\). Note that, since we do not have a Poisson
bracket on \(G\), we cannot ask if \(\Phi\) is a Lie algebra
homomorphism as we did with the comoment map of
Section~\ref{sec:symp-action}; we may however analogously ask if
\(\Phi_{\comm{X}{Y}} = \Omega(\rho_X,\rho_Y)\) for all \(X,Y\in
\fg\). The following result, along with
Proposition~\ref{prop:symp-cocycles-triv}, tells us that the action of
\(G\) on \((G,\Omega)\) is hamiltonian if and only if the Lie algebra
extension is trivial.

\begin{proposition}
  Where it exists, the comoment map \(\Phi\) can be chosen such that
  \(\Phi_{\comm{X}{Y}} = \Omega(\rho_X,\rho_Y)\) for all \(X,Y\in
  \fg\) if and only if \([c]=0\).
\end{proposition}

\begin{proof}
	By the analogous calculation to \eqref{eq:c-phi-rel}
	(or by Lemma~\ref{lemma:symp-cocycles}), we have
	\begin{equation}
		c(X,Y)=\Omega(\rho_X,\rho_Y)(g)-\Phi_{\comm{X}{Y}}=(\Ad^*_gc)(X,Y)-\Phi_{\comm{X}{Y}}(g).
	\end{equation}
	Clearly, if \(\Phi_{\comm{X}{Y}} = \Omega(\rho_X,\rho_Y)\) then \(c(X,Y)=0\). Conversely, if \(c=\partial_{\mathrm{CE}}\alpha\) for some \(\alpha\in\fg^*\) then we have
	\begin{equation}
		\Omega(\rho_X,\rho_Y) - \Phi_{\comm{X}{Y}}
		= c(X,Y) 
		= - \alpha(\comm{X}{Y}).
	\end{equation}
	We see that \(\Phi'\) defined by \(\Phi_X'(g)=\Phi_X(g)-\alpha(X)\)  satisfies both \(\imath_{\rho_X}\Omega=d\Phi'_X\) and \(\Phi'_{\comm{X}{Y}}=\Omega(\rho_X,\rho_Y)\).
\end{proof}

\begin{proof}[Proof of Theorem~\ref{thm:group-ext-symp-cocycle}]
  Let us choose a splitting of the algebra extension and fix notation
  as above. The calculations in the first part of the proof of
  Theorem~\ref{thm:homo-sym-ext} show that if \(\theta: G\to\fg^*\) is
  a symplectic group cocycle satisfying
  \(c(X,Y)=\pair{d_e\theta(X)}{Y}\), then the functions \(\Phi_X\in
  C^\infty\) given by \(\Phi_X(g)=-\pair{\theta(g)}{X}\) satisfy
  equation~\eqref{eq:neeb} of Theorem~\ref{thm:neeb}, and so by that
  theorem there exists a one-dimensional central extension
  \(\widehat{G}\) of \(G\) integrating the extension of Lie algebras.
  
  On the other hand, suppose that the group extension
  \eqref{eq:lie-group-ext} exists. Then by Theorem~\ref{thm:neeb}, for
  each element \(X\in\fg\) there exists a map \(\Phi_X\in C^\infty\)
  satisfying equation \eqref{eq:neeb}. We argued above that without
  loss of generality, these maps can be chosen such that
  \(\Phi_X(g)=-\pair{\theta(g)}{X}\) for some symplectic group cocycle
  \(\theta\) integrating \(c\).
\end{proof}
 
Note that by Proposition~\ref{prop:symp-cocycles-unique}, the symplectic
cocycle \(\theta\) appearing in the second part of the proof above is the
same as the one arising from the factored coadjoint action 
\(\widehat\Ad^*\) in equation~\eqref{eq:hat-coadjoint-theta}.
  
\subsection{Exact homogeneous symplectic spaces}
\label{sec:exact-symp}

Let \((M,\omega)\) be a homogeneous symplectic \(G\)-space for which 
\(\omega=d\eta\) for some \(\eta\in\Omega^1(M)\). In such a case, we call
\(\eta\) the \emph{symplectic potential}, and the pair \((M,\eta)\) is 
an \emph{exact homogeneous symplectic \(G\)-space}. We will now show
that a central extension of \(G\) satisfying the conclusion of
Theorem~\ref{thm:homo-sym-ext} can be constructed directly without appeal
to Neeb's Theorem~\ref{thm:neeb}, and that it takes a particularly simple form.

For an exact homogeneous symplectic $G$-space, the \(G\)-invariance
condition \(g^*\omega=\omega\) can be rewritten as
\begin{equation}
  d\qty(g^*\eta-\eta) = 0.
\end{equation}
If the one-form \(g^*\eta-\eta\) is exact for all \(g\in G\) (in
particular if \(M\) is simply-connected), for each \(g\in G\) there 
exists a function \(\Psi_g\in C^\infty(M)\) such that
\begin{equation}\label{eq:Psi-def}
  g^*\eta-\eta = d\Psi_g.
\end{equation}
We can assume without loss of generality that \(\Psi_g\) depends smoothly
on \(g\), in the sense that the function \(\tilde\Psi: G\times M\to\RR\)
given by \(\tilde\Psi(g,p)=\Psi_g(p)\) is smooth. Indeed, for any
\(X\in \fg\), evaluating \eqref{eq:Psi-def} on the fundamental vector
field \(\xi_X\) at a point \(p\in M\), we find
\begin{equation}
	\qty(\eL_{\xi_X}\Psi_g)(p) 
		= \eta(g_*\xi_X|_p)(gp) - \eta(\xi_X|_p)
		= \eta(\xi_{\Ad_g X})(gp) - \eta(p)
		=: \pair{F(g,p)}{X}
\end{equation}
where the last equality defines the smooth map \(F\in C^\infty(G\times M;\fg^*)\). 
It follows that \(\tilde\Psi\) is smooth up to the addition of a function 
\(f:G\to H^0(M) = \RR\), but this is exactly the freedom we have in the choice of
\(\Psi_g\) in equation~\eqref{eq:Psi-def}.

Taking \(\tilde\Psi\) to be smooth, the map \(\psi:\fg\to C^\infty(M)\),
\(X\mapsto\psi_X\) defined by \(\psi_X(p)=\dv{t}\Psi_{\exp(tX)}(p)|_{t=0}\)
satisfies
\begin{equation}\label{eq:psi-def}
	\eL_{\xi_X}\eta = d\psi_X,
\end{equation}
and so we can define a comoment map
\(\varphi:\fg\to C^\infty(M)\), \(X\mapsto\varphi_X\) by
\begin{equation}\label{eq:phi-def-psi}
	\varphi_x = \psi_X - \eta(\xi_X);
\end{equation}
indeed, we have
\begin{equation}
	d\varphi_X = d\psi_X  - d(\eta(\xi_X)) = \eL_{\xi_X}\eta - d\imath_{\xi_X}\eta = \imath_{\xi_X} d\eta = \imath_{\xi_X}\omega.
\end{equation}

\begin{lemma}\label{prop:gamma-const-cocycle}
	The map \(\gamma:G\times G\to C^\infty(M)\) given by
	\begin{equation}\label{eq:gamma-def}
		\gamma(g_1,g_2) := \Psi_{g_2} + g_2^*\Psi_{g_1} - \Psi_{g_1g_2}.
	\end{equation}
	takes values in the constant functions, and can thus be considered 
	as a (smooth) map \(\gamma:G\times G\to \RR\).
	Furthermore, \(\gamma\) satisfies the equation
	\begin{equation}\label{eq:gamma-cocycle}
		\gamma(g_2,g_3) + \gamma(g_1,g_2g_3) 
			= \gamma(g_1,g_2) + \gamma(g_1g_2,g_3)
	\end{equation}
	for all \(g_1,g_2,g_3\in G\).
\end{lemma}

\begin{proof}
For all \(g_1,g_2\), \((g_1g_2)^*\eta=g_2^* g_1^*\eta\), thus
\begin{equation}
\begin{split}
	d\gamma(g_1,g_2)
		= d\Psi_{g_2} + g_2^*d\Psi_{g_1} - d\Psi_{g_1g_2}
		= g_2^*\eta - \eta + g_2^*\qty(g_1^*\eta-\eta) - \qty(\qty(g_1g_2)^*\eta - \eta)
		= 0,
\end{split}
\end{equation}
so since \(M\) is connected, this shows that \(\gamma(g_1,g_2)\)
is constant. The second claim can be easily verified using the
definition of \(\gamma\).
\end{proof}

The last part of the above lemma says that \(\gamma\) is a cocycle in
\(C^2(G)=C^2(G;\RR)\). Recalling the defining equation \eqref{eq:Psi-def} for
\(\Psi\), we note that it is unique only up to addition of a map in
\(C^\infty(G)\); \(d\Psi'_g=d\Psi_g\) if and only if there exists some
\(f:G\to \RR\) such that \(\Psi'_g(p)=\Psi_g(p)+f(g)\) for all \(p\in M\),
\(g\in G\). The corresponding change in \(\gamma\) is
\begin{equation}
	\gamma'(g_1,g_2) = \gamma(g_1,g_2) + f(g_1) + f(g_2) - f(g_1g_2) = (g+\partial f)(g_1,g_2)
\end{equation}
where \(\partial\) denotes the differential of the group cochain complex,
so \([\gamma']=[\gamma]\in H^2(G)\). The action of \(G\) on \((M,d\eta)\)
thus determines an element of \(H^2(G)\), which itself corresponds to a
one-dimensional central extension
\begin{equation}
  \begin{tikzcd}
   1 \arrow[r] & \RR \arrow[r,] & \widehat{G} \arrow[r] &
   G \arrow[r] & 1.
  \end{tikzcd}
\end{equation}
See Appendix~\ref{sec:cent-ext-lie-group-coho} for more details. We can
describe this group extension explicitly, but let us first make a
convenient choice for \(\gamma\). We note that \(\Psi_e\) is constant by equation \eqref{eq:Psi-def} since \(M\) is connected, and
\(\gamma(g,e)=\gamma(e,g)=\Psi_e\) for all \(g\in G\) by
\eqref{eq:gamma-def}. Using the freedom to redefine \(\Psi\), we can 
assume without loss of generality that \(\Psi_e=0\), whence
\(\gamma(g,e)=\gamma(e,g)=0\). We then construct the group \(\widehat{G}\)
as follows: as a manifold, \(\widehat{G}=G\times \RR\), and the group
multiplication is given by
\begin{equation}
	(g_1,u_1)\cdot(g_2,u_2) = (g_1g_2,u_1+u_2+\gamma(u_1,u_2))
\end{equation}
with identity \(\widehat{e}=(e,0)\) and inverses
\((g,u)^{-1}=(g^{-1},-u-\gamma(g,g^{-1}))\). The morphisms in the short
exact sequence above are given by the natural inclusion and projection.

\begin{lemma}\label{lemma:theta-def-gamma}
	For all \(g\in G\) and \(X,Y\in\g\),
	\begin{align}
		\pair{\theta(g)}{X} 
		&= \dv{t}\qty(\gamma\qty(g,\exp(t\Ad_{g^{-1}}X))-\gamma(\exp(tX),g))\eval_{t=0},
			\label{eq:theta-gamma}\\
		c(X,Y) &= \pdv{s}\pdv{t}\qty(\gamma(\exp(sX),\exp(tY))
					-\gamma(\exp(sY),\exp(tX)))\eval_{s=0,t=0},\label{eq:c-gamma}
	\end{align}
	where \(\theta\) and \(c\) are the cocycles associated
	to the comoment map \(\varphi\) defined by
	equation~\eqref{eq:phi-def-psi}.
\end{lemma}

\begin{proof}
	For the first formula, we can use the definition of \(\gamma\),
	\eqref{eq:gamma-def}, to write
	\begin{align}
		\dv{t}\gamma(\exp(tX),g)\eval_{t=0} 
		&= g^*\psi_X 
			- \dv{t}\Psi_{\exp(tX)g}\eval_{t=0},\\
		\dv{t}\gamma(g,\exp(tY))\eval_{t=0} 
		&= \lambda(g)\psi_Y + \eL_{\xi_Y}\Psi_g 
			- \dv{t}\Psi_{g\exp(tY)}\eval_{t=0}.
	\end{align}
	We now note that \(\exp(tX)g=g\exp(t\Ad_{g^{-1}}X)\), so setting
	\(Y=\Ad_{g^{-1}}X\), taking the difference between these equations
	and using the definition of \(\theta\) gives
	\eqref{eq:theta-gamma}. The second formula follows from the first
	by differentiating and using the fact that
	\(\gamma(e,g)=0\).
\end{proof}

\begin{proposition}
	The central extension of \(G\) defined by
	\(\gamma\) integrates the central extension of \(\fg\) defined
	by \(c\), and thus suffices for the central extension \(\widehat{G}\)
	appearing in Theorem~\ref{thm:homo-sym-ext}, where we take
	\(K=\RR\).
\end{proposition}

\begin{proof}
	The adjoint action of \(\widehat{G}\) on \(\widehat\fg\) is given by
	\begin{equation}
	\begin{split}
		\widehat\Ad_{(g,u)}(Y,v) &= \dv{t}(g,u)\exp(t(Y,v))(g,u)^{-1} \eval_{t=0}\\
		&= \dv{t} (g,u)(\exp(tY),tv)\qty(g^{-1},-u-\gamma(g,g^{-1}))\eval_{t=0}\\
		&= \dv{t} \qty(g\exp(tY)g^{-1}, tv+\gamma(\exp(tY),g^{-1})-\gamma\qty(g,g^{-1})+\gamma\qty(g,\exp(tY)g^{-1}))\eval_{t=0},
	\end{split}
	\end{equation}
	and using the identity \(\exp(tY)g^{-1}=g^{-1}\exp(t\Ad_gY)\) and the
	cocycle condition for \(\gamma\), we find that
	\begin{equation}
		-\gamma\qty(g,g^{-1})+\gamma\qty(g,\exp(tY)g^{-1}) = \gamma(gg^{-1},\exp(t\Ad_gY)) - (g^{-1},\exp(t\Ad_gY))
	\end{equation}
	and so, using equation~\eqref{eq:theta-gamma},
	\begin{equation}\label{eq:hat-adjoint-theta-2}
		\widehat\Ad_{(g,u)}(Y,v) = (\Ad_gY,v -\theta(g^{-1})).
	\end{equation}
	Differentiating once again with respect to the group element, we find
	\begin{equation}
		\comm{(X,u)}{(Y,v)} = \widehat\ad_{(X,u)}(Y,v) = (\comm{X}{Y},c(X,Y)),
	\end{equation}
	so \(\widehat\fg\) is the Lie algebra of \(\fg\). We thus see that the
	short exact sequence expressing \(\widehat{G}\) as a central extension
	of \(G\) by \(\RR\) integrates the one expressing \(\widehat\fg\) as a
	central extension of \(\fg\).
\end{proof}

We note that equation~\eqref{eq:hat-adjoint-theta-2}
closely resembles \eqref{eq:hat-adjoint-theta}; in fact,
since \(\widehat\Ad_g=\widehat\Ad_{(g,u)}\) where the
former is the ``factored" coadjoint representation, we
they are in fact the same equation. We thus see that the
symplectic cocycle \(\theta\) arising from the moment map
is manifestly the one arising from the coadjoint action of
\(\widehat{G}\) (equation~\ref{eq:hat-adjoint-theta}); we
saw in Section~\ref{sec:interpretation-G-action} that the
same held in the general case (where \(\omega\) was not exact),
but there we needed to use a uniqueness result 
(Proposition~\ref{prop:symp-cocycles-unique}).

Finally, we can explicitly demonstrate the observation in
the remark following Theorem~\ref{thm:neeb} that the central
extension \(\widehat{G}\) integrating the algebra extension is
not unique, and that we have a choice of the kernel group.
For any integer \(n\), let \(\pi_n:\RR\to\RR/n\ZZ\cong S^1\)
be the canonical Lie group morphism and let
\((\pi_n)_*: C^\bullet(G)\to C^\bullet(G;S^1)\) be the
map \(\phi\to\pi_n\circ\phi\). One can easily check
that this commutes with the differential, inducing
maps \((\pi_n)_*: H^\bullet(G)\to H^\bullet(G;S^1)\).
In particular, \(\gamma_n:=\pi_n\circ\gamma\) is an
element in \(Z^2(G;S^1)\) and thus defines a
(geometrically trivial) central extension \(\widehat{G}_n\)
of \(G\) by \(S^1\). Now, if \(\widehat{G}\) is the
extension by \(\RR\) defined by \(\gamma\), letting
\(\Pi_n:\widehat{G}\to\widehat{G}_n\) be the map
\(\Pi_n(g,u)=(g,[u])\), we get the commutative diagram
with exact rows and covering maps for vertical arrows
\begin{equation}
\begin{tikzcd}
	& 1 \ar[r] & \RR \ar[r]\ar[d,"\pi_n"] & \widehat{G} \ar[r]\ar[d,"\Pi_n"] & G \ar[r]\ar[equal]{d} & 1\\
	& 1 \ar[r] & S^1 \ar[r] & \widehat{G}_n \ar[r] & G \ar[r] & 1.
\end{tikzcd}
\end{equation}
In particular, the lower row is also a central extension
of \(G\) integrating the Lie algebra extension. 

\appendix
\section{Lie group and Lie algebra cohomology}
\label{sec:cohomology}

Here we give a brief overview of the cohomology theories used in this paper.

\subsection{Group cohomology}
\label{sec:group-cohomology}
	 
Let \(G\) be a Lie group, \(\fg=\Lie G\) its Lie algebra and \(\rho:G\to\GL(V)\) a
representation of \(G\). The \emph{cochain complex for \(G\) with values in \((V,\rho)\)}
(or simply with values in \(V\) where there is no ambiguity) is the complex
\((C^\bullet(G;V),\partial)\) with cochain spaces
\begin{equation}
  C^p(G;V) = C^\infty(G^p,V),
\end{equation}
where \(G^p=\underbrace{G\times G\times \dots \times G}_{p\text{factors}}\), and 
differential \(\partial:C^p(G;V)\to C^{p+1}(G;V)\) given by
\begin{equation}\label{eq:grp-cpx-differential}
\begin{split}
(\partial\psi)(g_1,g_2,\dots,g_{p+1})
	= \rho({g_1})\psi(g_2,g_3,\dots,g_{p+1})
	&+\sum_{i=1}^p (-1)^i \psi(g_1,g_2,\dots,g_ig_{i+1},\dots,g_{p+1}) \\
	&\qquad + (-1)^{p+1}\psi(g_1,g_2,\dots,g_p).
\end{split}
\end{equation}
The cohomology \(H^\bullet(G;V)\) of this complex is known as the
\emph{cohomology of \(G\) with values in \(V\)}. In this paper, the group
cohomology theories we are concerned with take values either in the coadjoint
representation \((\fg^*,\Ad^*)\) or trivial representation \((\RR,1)\); for
the latter, we write \(C^\bullet(\fg):=C^\bullet(\fg;\RR)\) and
\(H^\bullet(\fg):=H^\bullet(\fg;\RR)\).

Now let \(A\) be a connected abelian Lie group whose group operation is 
denoted \(+\). We define \emph{the  cochain complex for \(G\) with values
in \(A\)} to be the complex \((C^\bullet(G;A),\partial)\), with cochain 
spaces
\begin{equation}
  C^p(G;A) = C^\infty(G^p,A)
\end{equation}
and differential \(\partial:C^p(G;A)\to C^{p+1}(G;A)\) given by the same
formula \eqref{eq:grp-cpx-differential}, except that the first term is
simply \(\psi(g_2,g_3,\dots,g_{p+1})\) (without the \(\rho({g_1})\) action).
Note that we do not require an action of \(G\) on \(A\). The cohomology
\(H^\bullet(G;A)\) of this complex is \emph{the cohomology of \(G\) with
values in \(A\)}. 

\subsection{Central extensions of Lie groups from group cohomology}
\label{sec:cent-ext-lie-group-coho}

Let \(G\) be a connected Lie group and \(A\) a connected abelian Lie group.
Given a cocycle \(\gamma\in C^2(G;A)\), one can define a Lie group structure 
on the product manifold \(G\times A\) by
\begin{equation}\label{eq:triv-grp-cent-ext-mult}
	(g_1,a_1)\cdot(g_2,a_2) = (g_1g_2,a_1+a_2+\gamma(g_1,g_2)).
\end{equation}
The cocycle condition for \(\gamma\) is equivalent to associativity
of this product, the identity element is
\begin{equation}
	\widehat{e} = (e,a_0)
\end{equation}
where \(a_0=-\gamma(g,e)=-\gamma(e,g)=-\gamma(e,e)\) -- the
latter equalities follow from the cocycle condition for \(\gamma\),
and inverses are given by
\begin{equation}\label{eq:inverse-formula}
	(g,a)^{-1} = (g^{-1},a_0-a-\gamma(g,g^{-1})).
\end{equation}
We denote \(G\times A\) equipped with this group structure by
\(\widehat{G}\). The map \(\iota: A\to G\times A\) given by
\(\iota(a)=(e,a_0+a)\) and the natural projection \(\pi: G\times A\to G\)
are Lie group homomorphisms, and they give us a short exact sequence
\begin{equation}
  \begin{tikzcd}
   1 \arrow[r] & A \arrow[r,"\iota"] & \widehat{G} \arrow[r,"\pi"] &
   G \arrow[r] & 1.
  \end{tikzcd}
\end{equation}
where the image of \(\iota\) is central; that is, \(\widehat{G}\) is a
central extension of \(G\) by \(A\). If a different cocycle
\(\gamma'\in C^2(G;A)\) is chosen, the central extensions they define are
equivalent if and only if \(\gamma'-\gamma\) is a coboundary. Thus we have
an injective map from \(H^2(G;A)\) to the set of equivalence classes of
central extensions of \(G\) by \(A\). Let us describe the image of this map. 

Any central extension of \(G\) by \(A\) is a fibre bundle (in
particular a principal \(A\)-bundle) over \(G\) with fibre \(A\), but
the central extension associated to an element in \(H^2(G;A)\) is
manifestly trivial (\(\widehat{G}\cong G\times A\)) as a fibre bundle.
We will call such an extension \emph{geometrically trivial}. Conversely,
given any geometrically trivial central extension, one can show that, in
any trivialisation, the group structure takes the form
\eqref{eq:triv-grp-cent-ext-mult} for some cocycle \(\gamma\in C^2(G;A)\).
If a different trivialisation is chosen, one obtains a different cocycle
\(\gamma'\) but \([\gamma']=[\gamma]\) in \(H^2(G;A)\), allowing us to
assign an element of \(H^2(G;A)\) to any geometrically trivial central
extension. This shows that the map from \(H^2(G;A)\) to the set of
geometrically trivial central extensions is a bijection.

\subsection{Chevalley--Eilenberg cohomology}
\label{sec:chevalley-eilenberg}

Let \(\rho: \fg \to \mathfrak{gl}(V)\) be a representation of the Lie
algebra \(\fg\) on the real vector space \(V\). The
\emph{(Chevalley--Eilenberg) cochain complex for \(\fg\) with values
  in \((V,\rho)\)} is the complex
\((C^\bullet(\fg;V),\partial_{\mathrm{CE}})\) with cochain spaces
\begin{equation}
\begin{split}
	C^p(\fg;V) = \Hom(\Wedge^p \fg,V)
\end{split}
\end{equation}
and differential \(\partial_{\mathrm{CE}}:C^p(\fg;V)\to C^{p+1}(\fg;V)\) given by
\begin{equation}
\begin{split}
  (\partial_{\mathrm{CE}}\psi)(X_1,X_2,\dots,X_{p+1})
  = &\sum_{i=1}^{p+1}(-1)^{i-1}\rho(X_i)\psi(X_1,X_2,\dots,\widehat{X_i},\dots,X_{p+1})\\ 
  &+ \sum_{i<j}(-1)^{i+j}\psi(\comm{X_i}{X_j},X_1,\dots,\widehat{X_i},\dots,\widehat{X_j},\dots,X_{p+1}),
\end{split}
\end{equation}
where the $\widehat{\phantom{X_i}}$ adorning a symbol denotes its
omission.  The cohomology \(H^\bullet(\fg;V)\) of this complex is the
\emph {(Chevalley--Eilenberg) cohomology of \(\fg\) with values in
  \((V,\rho)\)} (or simply with values in \(V\)).

We will be particularly interested in 2-cocycles of \(\fg\) with values in the
trivial representation. We write \(C^\bullet(\fg):=C^\bullet(\fg;\RR)\) and
\(H^\bullet(\fg):=H^\bullet(\fg;\RR)\) for cochains and cohomology in the trivial
representation.

\section{Structures on Lie groups}
\label{sec:lie-group-structures}

\subsection{Invariant vector fields}
\label{sec:inv-vfs}

Let \(G\) be a Lie group and \(\fg\) its Lie algebra, thought of as
the tangent space $T_eG$ at the identity. For each \(X\in\fg\), we let
\(\lambda_X,\rho_X\in\eX(G)\) be the left- and right-invariant vector
fields associated to \(X\) respectively:
\begin{align}
	&(\lambda_X)_g = (L_g)_*X= \dv{t} g\exp(tX)\eval_{t=0},
	&& (\rho_X)_g = (R_g)_*X =\dv{t} \exp(tX)g\eval_{t=0}
\end{align}
where \(L_g,R_g\) are left- and right- translation respectively. The maps
\(\lambda,\rho:\fg\to \eX(G)\) given by \(X\mapsto\lambda_X\) and 
\(X\mapsto\rho_X\) are injective, and
\begin{align}
	&\lambda_{\comm{X}{Y}} = \comm{\lambda_X}{\lambda_Y}
	&&\rho_{\comm{X}{Y}} = -\comm{\rho_X}{\rho_Y}.
\end{align}
for all \(X,Y\in\fg\); that is, \(\lambda\) is a Lie algebra
homomorphism and \(\rho\) is an anti-homomorphism. Furthermore,
composing these maps with evaluation at a point \(g\in G\), we find
that \(X\mapsto(\lambda_X)_g\) and \(X\mapsto(\rho_X)_g\) are
isomorphisms \(\fg\to T_gG\); in particular, the values of
left-invariant vector fields span every tangent space, as do those of
the right-invariant vector fields, and the sets of left- and right-
invariant vector fields both span \(\eX(G)\) as a
\(C^\infty(G)\)-module.

\begin{lemma}\label{lemma:livf-rivf-rels}
  For each \(X\in \fg\) and \(g\in G\), we have
  \begin{align}
    &(\lambda_X)_g = (\rho_{{\Ad_g} X})_g,
    &&(\rho_X)_g = (\lambda_{\Ad_{g^{-1}} X})_g.
  \end{align}
\end{lemma}

\begin{proof}
  We have
  \begin{equation}
    (\lambda_X)_g = (L_g)_*X = (R_g)_*(R_{g^{-1}})_*(L_g)_*X = (R_g)_*\Ad_g X = \rho_{\Ad_g X}
  \end{equation}
  and similarly for the other claim.
\end{proof}

\subsection{Left-invariant de Rham complex}
\label{sec:left-inv-forms}

For any differential form \(\omega\in\Omega^p(G)\), we have
\(L_g^*(d\omega)=d(L_g^*\omega)\); thus the De Rham complex
\((\Omega^\bullet(G),d)\) of \(G\) has a subcomplex
\((\Omega^\bullet_{\mathrm{LI}}(G),d)\) consisting of left-invariant
differential forms. Left-invariant forms are determined by their value
at the identity -- in particular \(\omega_g=L^*_{g^{-1}}\omega_e\) --
and these values lie in the Chevalley--Eilenberg cochain space
\(C^p(\fg)=\Wedge^p\fg^*\); furthermore one can show that
\(d_g\omega=L^*_{g^{-1}}\partial_{\mathrm{CE}}\omega_e\). It follows
from these observations that
\((\Omega^\bullet_{\mathrm{LI}}(G),d)\cong
(C^\bullet(\fg),\partial_{\mathrm{CE}})\). This observation will be
used in particular in Section~\ref{sec:action-to-extension}. In the
case that \(G\) is compact, the natural inclusion of the
left-invariant de Rham complex into the full de Rham complex induces
an isomorphism in cohomology
\(H^\bullet_{\mathrm{LI}}(G)\cong H_{\mathrm{dR}}^\bullet(G)\), so we have
\(H_{\mathrm{dR}}^\bullet(G)\cong H^\bullet(\fg)\).

\providecommand{\href}[2]{#2}\begingroup\raggedright\endgroup

\end{document}